\documentclass[11pt,a4paper,abstract=on]{scrartcl}

\usepackage[english]{babel}

\usepackage{lmodern}
\usepackage[T1]{fontenc}

\usepackage{ucs}
\usepackage[utf8x]{inputenc}

\usepackage[comma,numbers,sort&compress]{natbib}

\usepackage{amsfonts,amstext,amsmath,amssymb,amsopn,amsthm}

\usepackage{mathtools}

\usepackage{dsfont}
\usepackage{mathrsfs}

\usepackage{listings}
\usepackage{physics} 
\usepackage[dvipsnames, svgnames, x11names]{xcolor}

\usepackage{siunitx}
\usepackage{graphicx}
\usepackage{subfig}
\usepackage{wrapfig}

 \usepackage{todonotes}

\usepackage{csquotes}
\usepackage[colorlinks=true, 
linkcolor=blue,
pdfstartview=FitH,      
breaklinks=true,        
bookmarksopen=true,     
bookmarksnumbered=true  
]{hyperref}
\usepackage[capitalize,nameinlink]{cleveref}
\usepackage{url}
\usepackage{todonotes}
\usepackage{enumerate}
\usepackage{caption}

\usepackage[top=1in, bottom=1.5in, left=1in, right=1in]{geometry}
\usepackage[linesnumbered,lined,ruled, noend]{algorithm2e}

\newcommand{\paren}[1]{\left(#1\right)}
\newcommand{\ecklam}[1]{\left[#1\right]}

\renewcommand{\equiv}{\ensuremath{:=}}

\newcommand{\act}[1]{\left\langle {#1} \right\rangle}

\newcommand{\mymap}[3]{#1:\,#2 \to #3\,}

\newcommand{\mysetc}[2]{\left\{#1\,\middle|\,#2\right\}}
\newcommand{\set}[2]{\left\{#1\,\middle|\,#2\right\}}

\renewcommand{\ip}[2]{\left\langle #1,\, #2\right\rangle}

\newcommand{\cex}[2]{\ensuremath{\mathbb{E}\left[#1\,\middle|\,#2\right]}}
\newcommand{\cpr}[2]{\ensuremath{\mathbb{P}\left(#1\,\middle|\,#2\right)}}
\newcommand{\1}{\ensuremath{\mathds{1}} }

\newcommand{\icol}[1]{
  \left(\begin{smallmatrix}#1\end{smallmatrix}\right)%
}

\newcommand{\cb}{\ensuremath{\overline{\mathbb{B}}}}

\newcommand{\Nbb}{\mathbb{N}}

\newcommand{\Rn}{\mathbb{R}^n}
\newcommand{\Ecal}{\mathcal{E}}
\newcommand{\Pcal}{\mathcal{P}}

\newcommand{\xbar}{{\overline{x}}}
\newcommand{\hhbar}{\overline{h}}

\newcommand{\Fbar}{\overline{F}}
\newcommand{\ibar}{\overline{i}}

\DeclareMathOperator{\Supp}{supp}
\DeclareMathOperator{\dist}{dist}

\DeclareMathOperator*{\argmin}{\arg\!\min}

\DeclareMathOperator{\prox}{prox}

\DeclareMathOperator{\id}{Id}
\DeclareMathOperator{\Id}{Id}

\DeclareMathOperator{\Fix}{Fix}

\DeclareMathOperator{\diam}{diam}

\DeclareMathOperator{\inv}{inv}
\DeclareMathOperator{\indep}{\perp \!\!\! \perp\,}

\DeclareMathOperator{\Card}{Card}

\newcommand{\alphabar}{\overline{\alpha}}
\makeatletter
\def\@endtheorem{\endtrivlist\@endpefalse }
\makeatother

\newtheorem{thm}{Theorem}[section]
\newtheorem{cor}[thm]{Corollary}
\newtheorem{lemma}[thm]{Lemma}
\newtheorem{prop}[thm]{Proposition}
\theoremstyle{definition}
\newtheorem{example}[thm]{Example}
\newtheorem{definition}[thm]{Definition}
\newtheorem{assumption}[thm]{Assumption}

\newtheoremstyle{note}
{3pt}
{3pt}
{}
{}
{\bfseries}
{\bfseries :}
{.5em}
{}
\theoremstyle{note}

\newtheorem{rem}[thm]{Remark}

\title{Nonexpansive Markov Operators and Random Function Iterations for Stochastic Fixed Point Problems}
\author{Neal Hermer\thanks{Institute for Numerical and Applied Mathematics,
    University of Goettingen,
    37083 Goettingen, Germany. NH was supported by 
    Deutsche Forschungsgemeinschaft Research Training Grant 2088 TP-B5.
    E-mail:  \texttt{n.hermer@math.uni-goettingen.de}}, 
  D. Russell Luke\thanks{Institute for Numerical and Applied Mathematics,
    University of Goettingen,
    37083 Goettingen, Germany. DRL was supported in part by 
    Deutsche Forschungsgemeinschaft Research Training Grant 2088 TP-B5.
    E-mail:  \texttt{r.luke@math.uni-goettingen.de}}  
  and  Anja Sturm\thanks{Institute for Mathematical Stochastic,
    University of Goettingen,
    37077 Goettingen, Germany. AS was supported in part by Deutsche 
Forschungsgemeinschaft 
    Research Training Grant 2088 TP-B5.
    E-mail:  \texttt{asturm@math.uni-goettingen.de}}}
\date{\today}

\begin{document}
 \maketitle

 \begin{abstract}
   We study the convergence of random function iterations for finding
   an invariant measure of the corresponding Markov operator.  We call
   the problem of finding such an invariant measure the {\em
     stochastic fixed point problem}.  This generalizes earlier work
   studying the {\em stochastic feasibility problem}, namely, to find
   points that are, with probability 1, fixed points of the random
   functions \cite{HerLukStu19a}.  When no such points exist, the
   stochastic feasibility problem is called {\em inconsistent}, but
   still under certain assumptions, the more general stochastic fixed
   point problem has a solution and the random function iteration
   converges to an invariant measure for the corresponding Markov
   operator.  We show how common structures in deterministic fixed
   point theory can be exploited to establish existence of invariant
   measures and convergence in distribution of the Markov chain.  This framework
   specializes to many applications of current interest including, for
   instance, stochastic algorithms for large-scale distributed
   computation, and deterministic iterative procedures with
   computational error.  The theory developed in 
this study provides a solid basis for describing the convergence
of simple computational methods without the assumption of infinite precision
arithmetic or vanishing computational errors.  

\end{abstract}

{\small \noindent {\bfseries 2010 Mathematics Subject Classification:}
  Primary 60J05, 
  46N10, 
  46N30, 
  65C40, 
  49J55 
    Secondary  49J53,   
    65K05.\\ 
  }


\section{Introduction}
\label{sec:introduction}

Stochastic algorithms have emerged as a major approach to solving large-scale optimization problems.  The 
analysis of these algorithms is for the most part restricted to ergodic results, that is, convergence of the 
average of the iterates to a single fixed point.  The approach we present in this paper aims to provide a 
convergence theory for the entire distribution behind the iterates, not just their mean. In order to keep 
the already technical proofs as simple as possible, the underlying setting
is relatively benign:  we consider random selections of nonexpansive self-mappings on a 
Euclidean space, denoted $\Ecal$;  the self-mappings, $\mymap{T_i}{\Ecal}{\Ecal}$ where $i$
indexes a possibly uncountable collection $\{T_i\}_{i\in I}$, 
are understood as actions taken by 
an algorithm on an iterate $X_k$ to produce the next iterate $X_{k+1}$.  The procedure
is formally a {\em random function iteration} (RFI)\cite{Diaconis1999}.  

The present study is a continuation of a development begun in \cite{HerLukStu19a} which was 
confined to the assumption that the self-mappings have common fixed points. 
The deterministic analog to this situation is classical.  The deterministic result 
that tracks most closely to our assumptions is \cite[Theorem 4.1]{RuiLopNic15} where the authors study 
finite compositions of {\em firmly nonexpansive} self-mappings that have common fixed points, and which have 
{\em boundedly compact images}.  We are not aware of any convergence result for 
a deterministic fixed point iteration that does not require compactness of some sort, though the 
requirement of common fixed points can be dropped in Hilbert space settings.  Our stochastic 
results are most closely anticipated by Butnariu \cite{Butnariu95} who  
studied the variant where the mappings $T_i$ are projectors onto convex sets $C_i$;  the 
collection of sets in that study did not need to be countable nor have common points.  

We are concerned in this paper with, to a lesser extent,  (i) {\em existence of invariant
  distributions} of the Markov operators associated with the random
function iterations, and, our principal focus,  (ii) {\em convergence} of the Markov chain to an
invariant distribution. Rates of convergence and attendant stopping rules 
are postponed for a follow-up work.   
The existence theory is already well developed and is surveyed in Section \ref{sec:existence} 
below.  We show how existence is guaranteed when, for instance, the image is 
compact for some non-negligible
collection of operators $T_i$ (Proposition \ref{cor:finite_selection_existence}) or when 
the expectation of the 
random variables $X_k$ is finite (Proposition \ref{thm:ex_inRn}).  
Beyond these preparatory results, our main focus is convergence, 
and here there are two principal contributions:  
first, we provide convergence results without the assumption that the (possibly uncountably 
infinite collection of) self mappings
possess common fixed points; secondly, 
we prove convergence without any compactness assumption.  While the 
case of consistent stochastic feasibility can be analyzed with 
the standard tools from deterministic fixed point theory \cite{HerLukStu19a},
dropping the assumption of common fixed points requires the full extent of 
elementary probability theory;  the dividend for this is a vastly expanded 
range of applications.  That we do not require compactness for convergence shows that  
randomization of the usual deterministic strategies compensates for the absence of 
compactness, and indicates a way forward for other  results that rely on this assumption. 
The contribution to the theory of Markov chains is the extension to applications 
where the operators are not contractive, quasi- or otherwise.  

There are many more 
applications than one could reasonably list, but to reach the broadest possible audience
and inspired by \cite{Butnariu95}, 
a simple example from first
semester numerical analysis is illustrative.  Consider the
underdetermined linear system of equations
\[
 Ax=b, \quad A\in\mathbb{R}^{m\times n}, b\in \mathbb{R}^m, ~m<n.  
\]
Equivalent to this problem is the problem of finding the intersection of 
the hyperplanes defined by the single equations $\ip{a_j}{x}=b_j$ ($j=1,2,\dots. m$)
where $a_j$ is the $j$th row of the matrix $A$:
\begin{equation}\label{eq:ifeas}
 \mbox{Find }\quad \bar x\in \cap_{j=1}^m \{x~|~\ip{a_j}{x}=b_j\}.
\end{equation}
An intuitive technique to solve this problem is 
the method of cyclic projections:  Given an initial guess $x_0$, 
construct the sequence $(x_k)$ via
\begin{equation}\label{eq:CP}
 x_{k+1} = P_mP_{m-1}\cdots P_1 x_k,
\end{equation}
where $P_j$ is the orthogonal projection onto the $j$th hyperplane above. 
This method was proposed by von Neumann 
in \cite{Neumann50} where he also proved that, without numerical error, the iterates
converge to the projection of the initial point $x_0$ onto the 
itersection.

The projectors have a closed form representation, and the algorithm is 
easily implemented.  The results of one implementation for a randomly 
generated matrix $A$ and vector $b$ with $m=50$
and $n=60$ yields the following graph shown in Figure \ref{fig:Axb}(a). 
\begin{figure}[h]
  (a) \includegraphics[width=6cm,
  height=4cm]{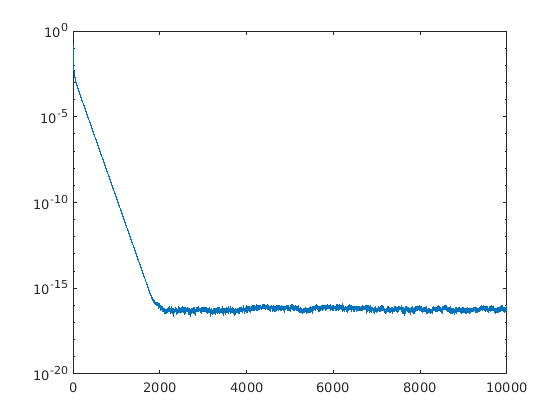} \hfill (b)
  \includegraphics[width=6cm,
  height=4cm]{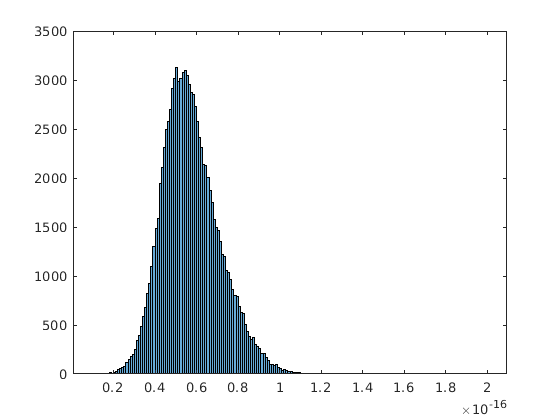}
\caption{(a) The residual $x_{k+1}-x_{k}$ of iterates of the cyclic projections
algorithm for solving the linear system $Ax=b$ for $A\in \mathbb{R}^{50\times 60}$,
and $b\in \mathbb{R}^{50}$ randomly generated.  (b) A histogram of the residual 
sizes over the last $8000$ iterations. }\label{fig:Axb}
\end{figure}
As the figure shows, the method performs as predicted by the theory, up 
to the numerical precision of the implementation.  After that point, the 
iterates behave more like random variables with distribution indicated by the 
histogram shown in Figure \ref{fig:Axb}(b).  The theory developed in 
this study provides a solid basis for describing the convergence
of simple computational methods without the assumption of infinite precision
arithmetic or vanishing computational errors \cite{Rockafellar76, SolodovSvaiter99}.    
This particular 
situation could be analyzed in the stability framework of perturbed convergent 
fixed point iterations with unique fixed points developed in \cite{ButReiZas07};
our approach captures their results and opens the way to a much broader range of 
applications.  
An analysis of nonmonotone fixed point iterations with error can be found already in 
\cite{IusPenSva03}, though the precision is assumed to increase quickly to exact evaluation. 

The main object of interest in our approach is the Markov operator on 
a space of probability measures with the appropriate metric.  We take for granted much of the basic theory of 
Markov chains, which interested readers can find, for instance, in \cite{HerLerLas} or \cite{MeyTwe}. 
We are indebted to the work of Butnariu and collaborators 
who studied stochastic iterative procedures for solving infinite dimensional linear operator 
equations
in \cite{Butnariu95, ButnariuIusemBurachik00, ButnariuFlam95,
ButnariuCensorReich97}. 
Another important application motivating our analytical strategy involves stochastic 
implementations of deterministic algorithms for large-scale optimization problems
\cite{BoydParikhChuPeleatoEckstein11, Combettes2018, Villa2019,
  Richtarik15, Nedic2011}.  Such stochastic algorithms are central to
distributed computation with applications in machine learning
\cite{BaldassiE7655, combettes2017a,DieuDurBac20, HardtMaRecht18, Richtarik16, Richtarik19}.  
Here each $T_{\xi_{k}}$ represents a randomly selected, low-dimensional update
mechanism in an iterative procedure.

As with classical fixed point iterations, the limit -- or more accurately,
{\em limiting  distribution} -- of the Markov chain, if it exists, will
in general depend on the initialization.  Uniqueness of invariant
measures of the Markov operator is not a particular concern for feasibility
problems where {\em any} feasible point will do.  
The notation and necessary 
background is developed in Section \ref{sec:consistent_feas_prob}, which 
we conclude with the main statements of this study (Section \ref{sec:mainres}).
Section \ref{sec:theory} contains the technical details, starting with 
existence theory in Section \ref{sec:existence}, general ergodic theory in  
Section \ref{sec:suppCvg} with gradually increasing regularity assumptions on the 
Markov operators, 
equicontinuity in Section \ref{sec:ergodicNonexp} and finally Markov operators
generated by nonexpansive mappings in Section \ref{sec:nexp}.  
The assumptions on the mappings generating the Markov operators are commonly 
employed in the analysis of deterministic algorithms in continuous optimization.   
Our first main 
result, Theorem \ref{cor:cesaroConvergenceRn},  establishes convergence for 
Markov chains that are generated from 
nonexpansive mappings in $\Ecal$ and follows easily in Section \ref{sec:setCvgNE} 
upon establishing tightness of the sequence of measures. 
Section \ref{sec:furtherprops} collects further facts needed for the 
second main result of this study, Theorem \ref{thm:a-firm convergence Rn}, 
which establishes convergence in the Prokhorov-L\`evy metric 
of Markov chains to an invariant measure (assuming this exists) when 
the Markov operators are constructed from 
{\em averaged mappings} in $\Ecal$ (Definition \ref{d:a-fne}).  
We conclude this study with Section \ref{sec:incFeas} where 
we focus on applications to optimization on measure spaces and 
(inconsistent) feasibility.  

\section{Random Function Iterations and the Stochastic Fixed Point Problem}
\label{sec:consistent_feas_prob}

In this section we give a rigorous formulation of the RFI, 
then interpret this as a Markov chain and define the
corresponding Markov operator. We then formulate modes of convergence
of these Markov chains to invariant measures for the Markov operators
and formulate the stochastic feasibility and stochastic fixed point problems.
At the end of this section we present the  main
results of this article.  The proofs of these results are developed in 
Section \ref{sec:theory}.  

Our notation is standard.  As usual, $\mathbb{N}$ denotes the natural
numbers {\em including} $0$.  
For the Euclidean space $\Ecal$ the Borel $\sigma$-algebra is denoted by 
$\mathcal{B}(\Ecal)$ and 
$(\Ecal,\mathcal{B}(\Ecal))$ is the 
corresponding measure space.  We denote by $\mathscr{P}(\Ecal)$ the set of all
probability measures on $\Ecal$.  The \emph{support of the probability measure}
$\mu$ is the smallest closed set $A$, for which $\mu(A)=1$  
and is denoted by $\Supp \mu$. 

There is a lot of overlapping notation in probability theory.  Where possible
we will try to stick to the simplest conventions, but the 
context will make certain notation preferable. 
The notation $X \sim \mu\in \mathscr{P}(\Ecal)$ means that the law
 of $X$, denoted $\mathcal{L}(X)$, satisfies 
 $\mathcal{L}(X)\equiv\mathbb{P}^{X} :=\mathbb{P}(X \in \cdot) = \mu$, 
 where $\mathbb{P}$ is the probability
 measure on some underlying probability space.  All of these different ways of 
 indicating a measure $\mu$ will be used.  

The open ball centered at $x\in \Ecal$ with
radius $r>0$ is denoted $\mathbb{B}(x,r)$; the closure of the ball is denoted
$\overline{\mathbb{B}}(x,r)$.  The distance of a point $x$ to a set 
$A\subset \Ecal$ is denoted by $d(x,A)\equiv \inf_{w\in A}\|x-w\|$. 
For the ball of radius $r$ around a
subset of points $A\subset \Ecal$, we write
$\mathbb{B}(A,r)\equiv \bigcup_{x\in A} \mathbb{B}(x,r)$.  The
$0$-$1$-indicator function of a set $A$ is given by
\[
\mathds{1}_A(x)=\begin{cases}1&\mbox{ if }x\in A,\\
                 0&\mbox{ else.}
                \end{cases}
\]

Continuing with the development initiated in the introduction, we will consider 
a collection of mappings $\mymap{T_{i}}{\Ecal}{\Ecal}$, $i \in I$
where $I$ is an arbitrary index set.    
The measure space of indexes is denoted by
$(I,\mathcal{I})$, and $\xi$ is an $I$-valued random variable on the
probability space $(\Omega,\mathcal{F},\mathbb{P})$. 
The pairwise independence of two random variables $\xi$ and $\eta$ is
denoted $\xi\indep\eta$.  The random variables $\xi_k$ in the sequence
$(\xi_{k})_{k\in\mathbb{N}}$ (abbreviated $(\xi_{k})$) 
are independent and identically
distributed (i.i.d.) \ with $\xi_{k}$ having the same distribution as
$\xi$ ($\xi_k\overset{\text{d}}{=} \xi$).  
The method of random
function iteration is formally presented in Algorithm \ref{algo:RFI}.
\begin{algorithm}    
\SetKwInOut{Output}{Initialization}
  \Output{Set $X_{0} \sim \mu_0 \in \mathscr{P}(\Ecal)$, $\xi_k\sim\xi\quad \forall k\in\Nbb$.}
    \For{$k=0,1,2,\ldots$}{
            {$ X_{k+1} = T_{\xi_{k}} X_{k}$}\\
    }
  \caption{Random Function Iteration (RFI)}\label{algo:RFI}
\end{algorithm}

\noindent We will use the notation
\begin{equation}\label{eq:X_RFI}
  X_{k}^{X_0} := T_{\xi_{k-1}} \ldots T_{\xi_{0}} X_{0}
\end{equation}
to denote the sequence of the RFI initialized with $X_0\sim \mu_0$.
This is particularly helpful when characterizing sequences initialized
with the delta distribution of a point, where
$X_{k}^{x}$ denotes the RFI sequence initialized
with $X_0\sim \delta_x$. The following assumptions will be employed
throughout.

\begin{assumption}\label{ass:1}
  \begin{enumerate}[(a)]
  \item\label{item:ass1:indep} $X_{0},\xi_{0},\xi_{1}, \ldots, \xi_{k}$
    are independent for every $k \in \mathbb{N}$, where
    $\xi_{k}$ are i.i.d. for all $k$ with the same distribution as $\xi$. 
  \item\label{item:ass1:Phi} The function $\mymap{\Phi}{\Ecal\times
      I}{\Ecal}$, $(x,i)\mapsto T_{i}x$ is measurable.
  \end{enumerate}
\end{assumption}

\subsection{RFI as a Markov chain}
\label{sec:SPMasMC}
Markov chains are conveniently defined in terms of {\em transition kernels}.
A  transition kernel is a mapping 
$\mymap{p}{\Ecal\times \mathcal{B}(\Ecal)}{[0,1]}$ that is measurable in the first
argument and is a probability measure in the second argument; 
that is, \ $p(\cdot,A)$ is measurable for all $A \in
\mathcal{B}(\Ecal)$  and
$p(x,\cdot)$ is a probability measure for all $x \in \Ecal$.
\begin{definition}[Markov chain]
  A sequence of random variables $(X_{k})$,
  $\mymap{X_{k}}{(\Omega,\mathcal{F},\mathbb{P})}{(\Ecal,\mathcal{B}(\Ecal))}$
  is called Markov chain with transition kernel $p$ if for all $k \in
  \mathbb{N}$ and $A \in \mathcal{B}(\Ecal)$
  $\mathbb{P}$-a.s.\ the following hold:
  \begin{enumerate}[(i)]
  \item $\cpr{X_{k+1} \in A}{X_{0}, X_{1}, \ldots, X_{k}} =
    \cpr{X_{k+1} \in A}{X_{k}}$;
  \item $\cpr{X_{k+1} \in A}{X_{k}} = p(X_{k},A)$.
  \end{enumerate}
\end{definition}

\begin{prop}\label{thm:RFIasMC}
  Under Assumption \ref{ass:1},
  the sequence of random variables $(X_{k})$
  generated by Algorithm \ref{algo:RFI} is a Markov chain with transition
  kernel $p$ given by 
\begin{equation}\label{eq:trans kernel}
  (x\in \Ecal) (A\in
  \mathcal{B}(\Ecal)) \qquad p(x,A) \equiv \mathbb{P}(\Phi(x,\xi) \in A) =
  \mathbb{P}(T_{\xi}x \in A)
\end{equation}
for the measurable \emph{update function} 
$\mymap{\Phi}{\Ecal\times  I}{\Ecal}$, $(x,i)\mapsto T_{i}x$.
\end{prop}
\begin{proof}
It follows from \cite[Lemma 1.26]{kallenberg1997} that the 
mapping  $p(\cdot,A)$ defined by  \eqref{eq:trans kernel} is measurable for all 
$A \in\mathcal{B}(\Ecal)$,  and it is immediate from the definition that 
$p(x,\cdot)$ is
a probability measure for all $x \in \Ecal$.  So $p$ defined by 
\eqref{eq:trans kernel} 
is a transition kernel.  The remainder of the 
statement is an immediate consequence of the disintegration theorem 
(see, for example, \cite{stroock2010probability}). 
\end{proof}

The Markov operator $\mathcal{P}$ is defined pointwise for a measurable 
function 
$\mymap{f}{\Ecal}{\mathbb{R}}$ via
\begin{align*}
  (x\in \Ecal)\qquad \mathcal{P}f(x):= \int_{\Ecal} f(y) p(x,\dd{y}),
\end{align*}
when the integral exists. Note that
\begin{align*}
  \mathcal{P}f(x) = \int_{\Ecal} f(y)
  \mathbb{P}^{\Phi(x,\xi)}(\dd{y}) = \int_{\Omega}
  f(T_{\xi(\omega)}x) \mathbb{P}(\dd{\omega})= \int_{I}
  f(T_{i}x) \mathbb{P}^{\xi}(\dd{i}).
\end{align*}

The Markov operator $\mathcal{P}$ 
is \emph{Feller} if $\mathcal{P}f \in C_{b}(\Ecal)$ whenever $f \in
C_{b}(\Ecal)$, where $C_{b}(\Ecal)$ is the set of bounded and continuous
functions from $\Ecal$ to $\mathbb{R}$. 
This property is central to the theory of existence of invariant measures 
introduced below. 
The next fundamental result establishes the relation of the Feller property 
of the Markov operator to the generating mappings $T_{i}$.  
\begin{prop}[Theorem 4.22 in \cite{Bellet2006}]
\label{thm:Feller}
  Under Assumption \ref{ass:1}, if $T_{i}$ is continuous for all $i\in
  I$, then the Markov operator $\mathcal{P}$ is Feller.
\end{prop}

Let $\mu\in \mathscr{P}(\Ecal)$.  In a slight abuse of notation we denote the
dual Markov operator $\mymap{\mathcal{P}^{*}}{\mathscr{P}(\Ecal)}{\mathscr{P}(\Ecal)}$ 
acting on a measure $\mu$ by action on 
the right by $\mathcal{P}$ via
\begin{align*}
  (A \in \mathcal{B}(\Ecal))\qquad (\mathcal{P}^{*}\mu) (A):=
  (\mu\mathcal{P}) (A) := \int_{\Ecal} p(x,A) \mu(\dd{x}).
\end{align*}
This notation allows easy identification of the distribution of the $k$-th 
iterate of the Markov chain generated
by Algorithm \ref{algo:RFI}: $\mathcal{L}(X_{k}) =
\mu_0 \mathcal{P}^{k}$.

\subsection{The Stochastic Fixed Point Problem}\label{sec:consist RFI}
As studied in \cite{HerLukStu19a}, the {\em stochastic feasibility}
problem is stated as follows:
\begin{align}
  \label{eq:stoch_feas_probl}
\mbox{ Find }  x^{*} \in C := \mysetc{x \in \Ecal}{\mathbb{P}(x = T_{\xi}x) = 
1}.
\end{align}
A point $x$ such that $x = T_{i}x$ is a {\em fixed point} of the operator $T_{i}$;
the set of all such points is denoted by 
\begin{align*}
  \Fix T_{i} \equiv \mysetc{x \in \Ecal}{x = T_{i}x}.
\end{align*}
In \cite{HerLukStu19a} it was assumed that $C\neq\emptyset$.  If
$C=\emptyset$ this is called the {\em inconsistent stochastic
  feasibility} problem.  
  
Inconsistent stochastic feasibility is far from exotic. 
Take, for example, the not unusual assumption of additive noise:
define $T_{\eta}(x)\equiv f(x) + \eta$ where $\mymap{f}{\Rn}{\Rn}$ and 
$\eta$ is noise.  Then 
$\mathbb{P}(T_\eta(x) = x) = \mathbb{P}(\eta = x-f(x)) = 0$.  More concretely,
let $f=\Id - t\nabla F$ where $\mymap{F}{\Rn}{\mathbb{R}}$ is a differentiable, strongly
convex function and $t$ is some appropriately small stepsize.  This yields the noisy
gradient descent method 
\[
 T_\eta(x) = x - t\nabla F(x)+\eta. 
\]
Though this has no fixed point, the additive noise $\eta$ can be constructed so that the 
resulting Markov chain is ergodic and its (unique!) invariant distribution concentrates around
the unique global minimum of $F$. This example is purely for illustration.  
The stochastic gradient or stochastic approximation technique traces its origins
to the seminal paper of Robbins and Monro \cite{RobbinsMonro51}.  This has been 
studied as a purely stochastic method for stochastic optimization in \cite{Pflug86,GeiPfl19};
more recent studies include \cite{DieuDurBac20}.  We emphasize here that our results 
{\em do not} address the issue of convergence of limiting distributions for the RFI as
the step size vanishes, $t\to 0$.  

The inconsistency of the problem formulation is an artifact of asking
the wrong question.  
A fixed point of the (dual) Markov operator $\mathcal{P}$
is called an \emph{invariant} measure, i.e.\ $\pi \in \mathscr{P}(\Ecal)$
is invariant whenever $\pi \mathcal{P} = \pi$.  The
set of all invariant probability measures is denoted by 
$\inv \mathcal{P}$.
We are interested in the
following generalization of \eqref{eq:stoch_feas_probl}:
\begin{align}
  \label{eq:stoch_fix_probl}
  \mbox{Find}\qquad \pi\in\inv\mathcal{P}.
\end{align}
We refer to this as the {\em stochastic fixed
  point problem.}  


\subsection{Modes of convergence}
\label{sec:modesofcvg}
In \cite{HerLukStu19a}, we considered
almost sure convergence of the sequence $(X_k)$ to a random variable
$X$:
\[  
X_k\to X \text{ a.s.} \quad \text{as } k \to \infty.
\]
Almost sure convergence is commonly encountered
in the studies of stochastic algorithms in optimization, 
and can be guaranteed for consistent stochastic feasibility 
under most of the regularity assumptions on $T_i$ considered 
here (see \cite[Theorem 3.8 and 3.9]{HerLukStu19a}) though this
does not require the full power of the theory of general Markov
processes.   In fact, the 
next result shows that almost sure convergence is possible only
for consistent stochastic feasibility.
The following statement first appeared in Lemma 3.2.1 of \cite{Hermer2019}.  
\begin{prop}[a.s. convergence implies consistency]\label{thm:Hermer} 
Let $\mymap{T_i}{\Ecal}{\Ecal}$ 
be continuous for all $i\in I$.  
Let $\pi\in \inv\mathcal{P}\neq\emptyset$ 
and $X_0\sim\pi$.  Generate the sequence $(X_k^{X_0})_{k\in\Nbb}$
via Algorithm \ref{algo:RFI} where $X_0\indep \xi_k$ for all $k$.  
If the sequence $(X_k^{X_0})$ converges almost surely, 
then the stochastic feasibility problem is consistent.   Moreover, 
$\Supp\pi\subset C$.   
\end{prop}
Before proceeding to the proof, note that the {\em measure} remains 
the same for each iterate $X_k^{X_0}$ -  the issue here is when the 
{\em iterates} converge (almost surely).  
\begin{proof}
In preparation for our argument, which is by contradiction, choose any 
$x\in\Supp\pi$ where 
$\pi\in\inv\mathcal{P}$,  and define 
 \[
 I^\epsilon\equiv \set{i\in I}{\norm{T_ix-x}>\epsilon}
 \]
 for $\epsilon\geq 0$.   Note that $I^\epsilon\supseteq I^{\epsilon_0}$ 
whenever 
 $\epsilon\leq \epsilon_0$.  Define the set
\[
 J^\epsilon_\delta\equiv \set{i\in I}{\norm{T_ix- T_iy}\leq \epsilon, \quad 
 \forall y\in \overline{\mathbb{B}}\paren{x, \delta}}.
\]
These sets 
satisfy $ J^\epsilon_{\delta_1}\subset J^\epsilon_{\delta_2}$ whenever 
$\delta_1\geq \delta_2$
and, since $T_i$ is continuous for all $i\in I$, we have that for each $\epsilon>0$,  
$ J^\epsilon_\delta\uparrow I$ as $\delta\to 0$.
A short argument shows 
that $I^\epsilon$ and $J^\epsilon_\delta$ are measurable for each $\delta$ and  $\epsilon>0$.  
 
Suppose now, to the contrary, that $C=\emptyset$.  
Then $\mathbb{P}(T_\xi x=x)<1$ and hence 
$\mathbb{P}(\norm{T_\xi x- x}>0)>0$.    Since $I^\epsilon\supseteq I^{\epsilon_0}$ 
for 
 $\epsilon\leq \epsilon_0$ 
we have $\mathbb{P}^\xi(I^{\epsilon_0})\leq \mathbb{P}^\xi(I^{\epsilon})$
 whenever $\epsilon\leq \epsilon_0$.  
In particular, there must exist an $\epsilon_0$ such that 
$0<\mathbb{P}^\xi(I^{\epsilon_0})$.  On the other hand, 
there is a constant $\delta>0$ such that $\delta< \epsilon_0/2$ and 
$\mathbb{P}^\xi(K^{\epsilon_0}_\delta)>0$
where $K^{\epsilon_0}_\delta\equiv I^{\epsilon_0}\cap J^{\epsilon_0/2}_\delta$. 
 This construction then yields
\[
(\forall i\in K^{\epsilon_0}_\delta)\qquad \norm{T_iy-x}\geq  \norm{T_ix-x} -  
\norm{T_iy-T_ix}\geq \frac{\epsilon_0}{2}>\delta
 \quad\forall y\in \overline{\mathbb{B}}(x,\delta).   
\]

Next, we claim that $X_k^{X_0}\sim\pi$ for all 
$k\in\Nbb$.  Indeed, for any $Y \sim \pi$, if 
  $\xi$ is independent of $Y$, then $T_{\xi} Y \sim \pi$.
To see this, note that 
For $A \in \mathcal{B}(\Ecal)$ 
Fubini's Theorem and disintegration 
yield   
  \begin{align*}
    \mathbb{P}(T_{\xi}Y \in A) &= \mathbb{E}[
    \cex{\1_{A}(T_{\xi}Y)}{\xi}] = \mathbb{E} \int \1_{A}(T_{\xi}y)
    \pi(\dd{y}) = \int \int \1_{A}(z) \mathbb{P}^{T_{\xi}y}(\dd{z})
    \pi(\dd{y}) \\ &= \int \int \1_{A}(z) p(y,\dd{z}) \pi(\dd{y}) =
    \pi\mathcal{P}(A) = \pi(A) = \mathbb{P}(Y \in A).
  \end{align*}
It follows that $\Supp \mathcal{L}(Y) =
\Supp \mathcal{L}(T_{\xi} Y)$, and since $\xi_k$ are i.i.d, 
 $X_{k}^{Y} \sim \pi$ for all $k \in \mathbb{N}$. This establishes
 the claim. 

The 
independence of $\xi_k$ and $X_0$ for all $k$ implies the independence of 
$\xi_k$ and $X_k^{X_0}$ for all $k$.  Moreover,  
$\mathbb{P}\paren{X_k^{X_0} \in  B_\delta} = \pi(B_\delta) > 0$
for all $k\in \Nbb$, where to avoid clutter we denote $B_\delta\equiv 
\mathbb{B}(x,\delta)$.  
This yields
\[
(\forall k\in \Nbb)\quad  \mathbb{P}\paren{X^{X_0}_k\in B_\delta, 
X^{X_0}_{k+1}\notin B_\delta} \geq 
 \mathbb{P}\paren{X^{X_0}_k\in B_\delta, \xi_k\in K^{\epsilon_0}_\delta} = 
 \pi(B_\delta)\mathbb{P}^\xi(K^{\epsilon_0}_\delta)>0.
\]
Thus, we also have 
\begin{eqnarray*}
(\forall k\in \Nbb)\quad  \mathbb{P}\paren{X^{X_0}_k\in B_\delta, 
X^{X_0}_{k+1}\in B_\delta} &=& 
\mathbb{P}\paren{X^{X_0}_k\in B_\delta}-\mathbb{P}\paren{X^{X_0}_k\in B_\delta, 
X^{X_0}_{k+1}\notin B_\delta} \\
&\leq&  \pi(B_\delta) -  \pi(B_\delta) \mathbb{P}^\xi(K^{\epsilon_0}_\delta)\\
&=&  \pi(B_\delta) (1-  \mathbb{P}^\xi(K^{\epsilon_0}_\delta)) < \pi(B_\delta).
\end{eqnarray*}
However, by assumption, $X_k^{X_0}\to X_*$ a.s. for some random variable $X_*$ with 
$\mathbb{P}\paren{X_* \in  B_{\delta}} = \pi(B_{\delta}) > 0.$ If  $X_* \in  B_{\delta}$
then due to the a.s. convergence there exists a (random) $k_*$ such that $X^{X_0}_{k}\in B_\delta$
for all $k\geq k_*.$ This implies that 
\[
  \mathbb{P}\paren{X^{X_0}_k\in B_\delta, X^{X_0}_{k+1}\in B_\delta, X_* \in  B_{\delta} }\to 
   \mathbb{P}\paren{X_* \in  B_{\delta}} = \pi(B_{\delta}).
\]
which is a contradiction since by the above
\[
  \mathbb{P}\paren{X^{X_0}_k\in B_\delta, X^{X_0}_{k+1}\in B_\delta, X_* \in  B_{\delta} }\leq 
   \mathbb{P}\paren{X^{X_0}_k\in B_\delta, X^{X_0}_{k+1}\in B_\delta}< \pi(B_\delta).
\]
So it must be true that $\mathbb{P}\paren{\norm{T_\xi x-x}>0}=0$.  In other 
words, $\mathbb{P}\paren{T_\xi x=x}=1$, hence $C\neq\emptyset$.    
Moreover, since the point $x$ was any arbitrary point in 
$\Supp\pi$, we conclude that 
$\Supp\pi\subset C$. 
\end{proof}

For inconsistent feasibility, or more general 
stochastic fixed point problems that are the aim of the RFI,  Algorithm \ref{algo:RFI}, 
we focus on {\em convergence in distribution}.
Let $(\nu_{k})$ be a sequence of probability measures on $\Ecal$.  The sequence 
$(\nu_{k})$ is said to converge in distribution to $\nu$ whenever $\nu \in
\mathscr{P}(\Ecal)$ and for all $f \in C_{b}(\Ecal)$ it holds that $\nu_{k} f
\to \nu f$ as $k \to \infty$, where $\nu f := \int f(x) \nu(\dd{x})$. 
Equivalently a sequence of random variables $(X_{k})$ is said to converge in 
distribution if their laws $(\mathcal{L}(X_{k}))$ do.

We now consider two modes of 
convergence in distribution for the corresponding sequence of measures 
$(\mathcal{L}(X_{k})))_{k \in \mathbb{N}}$ 
on 
$\mathscr{P}(\Ecal)$:
\begin{enumerate}
\item convergence in distribution of the Ces\`{a}ro averages of 
 the  sequence $(\mathcal{L}(X_{k}))$ to a probability measure 
  $\pi \in \mathscr{P}(\Ecal)$, i.e.\ for any $f \in C_{b}(\Ecal)$
  \begin{align*}
    \nu_{k}f:= \frac{1}{k} \sum_{j=1}^{k} \mathcal{L}(X_{j}) f =
    \mathbb{E}\left[\frac{1}{k}\sum_{j=1}^{k} f(X_{j})\right] \to \pi
    f, \qquad \text{as } k \to \infty; 
  \end{align*}
\item convergence in distribution 
of the sequence 
  $(\mathcal{L}(X_{k}))$ to a probability measure $\pi \in \mathscr{P}(\Ecal)$, i.e.\
  for any $f \in C_{b}(\Ecal)$
  \begin{align*}
    \mathcal{L}(X_{k})f = \mathbb{E}[f(X_{k})] \to \pi f, 
    \qquad \text{as } k \to \infty.
  \end{align*}
\end{enumerate}
Clearly, the second mode of convergence implies the first.  
This is used in Section \ref{sec:setCvgNE} and Section \ref{sec:cvgAverages}.

An elementary fact from the theory of Markov chains
(Proposition \ref{thm:construction_inv_meas}) is that, if the Markov operator 
$\mathcal{P}$ is Feller and $\pi$ is a cluster
point of $(\nu_k)$ with respect to convergence in distribution 
then $\pi$ is an invariant probability measure.  
Existence of invariant 
measures for a Markov operator then amounts to verifying that the 
operator is Feller (by Proposition \ref{thm:Feller}, automatic if the $T_i$ are 
continuous) and that cluster points exist (guaranteed by 
{\em tightness}  -- or compactness with respect to the topology of convergence in 
distribution -- of the sequence, see \cite[Section 5]{Billingsley}.
In particular, this means that 
there exists a convergent subsequence $(\nu_{k_{j}})$ with
\begin{align*}
  (\forall f\in C_b(\Ecal))\qquad \nu_{k_{j}} f=
  \mathbb{E}\left[\frac{1}{k_{j}}\sum_{i=1}^{k_{j}} f(X_{i})\right] \to \pi f,
  \qquad \text{as } j \to \infty.
\end{align*}
Convergence of the whole sequence, i.e.\ $\nu_{k} \to \pi$, amounts
then to showing that $\pi$ is the unique cluster point of $(\nu_{k})$
(see Proposition \ref{thm:cvg_subsequences}).

Common metrics for spaces of measures are the 
{\em Prokhorov-L\`evy distance} and the {\em Wasserstein metric}.  
  \begin{definition}[Prokhorov-L\`evy\&Wasserstein distance]\label{d:PL}
  Let $(G,d)$ be a separable complete metric space and let $\mu, \nu \in 
\mathscr{P}(G)$.
\begin{enumerate}[(i)]
\item The Prokhorov-L\`evy distance, denoted by $d_P$, is defined by
  \begin{equation}\label{eq:PL}
    d_{P}(\mu,\nu) = \inf\mysetc{\epsilon >0}{\mu(A) \le
      \nu(\mathbb{B}(A,\epsilon)) + \epsilon,\, \nu(A)\le
      \mu(\mathbb{B}(A,\epsilon))+\epsilon \quad \forall A \in
      \mathcal{B}(G)}.
  \end{equation} 
  \item  For $p \geq 1$ let 
  \begin{equation}\label{eq:p-probabiliy measures}
       \mathscr{P}_{p}(G) = \mysetc{\mu \in \mathscr{P}(G)}{ \exists\, x
      \in G \,:\, \int d^{p}(x,y) \mu(\dd{y}) < \infty}.
  \end{equation}
  The Wasserstein $p$-metric on $\mathscr{P}_{p}(G)$, denoted $W_p$, is defined by 
 \begin{equation}\label{eq:Wasserstein}
  W_p(\mu, \nu)\equiv \paren{\inf_{\gamma\in C(\mu, \nu)}\int_{G\times G} 
d^p(x,y)\gamma(dx, dy)}^{1/p}\quad (p\geq 1)
 \end{equation}
 where $C(\mu, \nu)$ is the set of couplings of $\mu$ and $\nu$ 
 (measures on the product space $G\times G$ whose marginals
 are $\mu$ and $\nu$ respectively - see \eqref{eq:couplingsDef}).  
 \end{enumerate}
  \end{definition}
For probability measures on separable metric spaces the Prokhorov-L\`evy distance metrizes 
weak convergence (convergence in distribution).  Convergence in Wasserstein distance is stronger 
than convergence in distribution (the $p$-th moments of the measures also converge in the $p$-Wasserstein metric). 
Indeed, it can be seen from the metrics that 
\[
 d_P(\mu,\nu)^2\leq W_p(\mu,\nu)^p\quad (p\geq 1).
\]

\subsection{Regularity}
Our main results concern convergence of Markov chains under common regularity assumptions 
on the mappings $\{T_{i}\}$.  The regularity
of $T_{i}$ is dictated by the application, and our primary interest is to follow
this through to the regularity of the corresponding Markov operator.  In 
\cite{LukNguTam18} a framework was developed for a quantitative 
convergence analysis of set-valued mappings $T_{i}$ that are {\em calm} 
(one-sided Lipschitz continuous in the sense of set-valued-mappings)  
with Lipschitz constant possibly greater than 1.  Since we restrict our attention to 
nonexpansive operators (Lipschitz constant 1), we do not need the added complication of 
set-valued mappings, but this theory can be extended to such settings. 

The definitions below are simplified versions of the analogous properties in more general 
settings (see \cite{BLL, LauLuk21}.)
\begin{definition}[pointwise nonexpansive and averaged mappings in $\Ecal$]
\label{d:a-fne} 
Let  $\mymap{F}{D}{\Ecal}$ where $D\subset\Ecal$.
  \begin{enumerate}[(i)]
  \item The mapping $F$  is said to be \emph{pointwise  
  nonexpansive at $x_0\in D$ on $D$} 
whenever  
\begin{equation}\label{eq:pane}
\norm{F(x)-F(x_0)} \le \,\norm{x-x_0}, 
\qquad \forall x \in D.
\end{equation}
    When the above inequality holds for all $x_0\in D$ then $F$ is said to be 
    {\em nonexpansive on $D$}.   
  \item The mapping $F$ is said to be {\em pointwise averaged at $x_0\in D$ on $D$} 
  whenever
\begin{eqnarray}
&&
	\exists \alpha\in (0,1):\nonumber\\
	  \label{eq:paafne}&&	\quad	\norm{F(x)-F(x_0)}^2 \le  \norm{x-x_0}^2 - 
      \tfrac{1-\alpha}{\alpha}\psi(x,x_0, F(x), F(x_0))\quad \forall x \in D
		\nonumber
\end{eqnarray}
where the {\em transport discrepancy} $\psi$ of  $F$ at $x, x_0$, $F(x)$ and $F(x_0)$
is defined by  
\begin{equation}
\label{eq:nice ineq}
\psi(x,x_0, F(x), F(x_0))\equiv \norm{(F(x)- x)-(F(x_0)- x_0)}^2.
\end{equation}
When the above inequality holds for all $x_0\in D$ then $F$ is said to be 
{\em averaged on $D$}. 
  \end{enumerate}
\end{definition}

Our definition of averaged mappings differs from the standard definition in 
that it places the transport discrepancy $\psi$ at the center.  Equation 
\eqref{eq:nice ineq} shows, at least in a Hilbert space setting, 
that $\psi$ is nonnegative.  Consequently, 
any pointwise averaged mapping on $\Ecal$ is also pointwise nonexpansive.  
This is not the case in more general metric spaces.  

Our definition of averaged mappings with $\alpha=1/2$ is equivalent to the 
definition of {\em firmly contractive} mappings given in \cite[Definition 6]{Browder67}.
We conform to the dominant terminology in linear spaces originating with \cite{BaiBruRei78}. 
We note, however, that the descriptor ``averaging'' is a misnomer since the same notion
plays a central role in general metric spaces where the concept of 
averaging is vacuous; in metric space settings such mappings are called 
{\em $\alpha$-firmly nonexpansive} \cite{AriLeuLop14, BLL, LauLuk21}.

In normed linear spaces, Baillon and Bruck \cite{Baillon1996} showed that 
nonexpansive mappings whose orbits are bounded under 
convex relaxations are {\em asymptotically regular} with a universal rate constant.  
Precisely: let $\mymap{T}{D}{D}$ be nonexpansive, where $D$ is a convex subset 
of a normed linear space, and define $x_m$ recursively by 
$x_m= T_\lambda x_{m-1}\equiv \left((1-\lambda)\Id + \lambda T\right)$
for $\lambda \in (0,1)$ and $x_0\in D$.   
If $\|x - TT_\lambda^k x\|\leq 1$ for any $\lambda\in(0,1)$ and for all $0\leq k\leq m$, then    
\cite[Main Result]{Baillon1996}
\begin{equation}\label{eq:universal 1}
\|x_m - Tx_m\|<\frac{\diam{D}}{\sqrt{\pi m\lambda(1-\lambda)}}. 
\end{equation}

Cominetti, Soto and Vaisman \cite{CominettiSotoVaisman} recently confirmed
a conjecture of Baillon and Bruck that a universal rate constant also holds for 
nonexpansive mappings with arbitrary relaxation in $(0,1)$ chosen at each iteration; 
in particular, that 
\[
\|x_m - T x_{m}\|\leq\frac{\diam{D}}{\sqrt{\pi\sum_{k=1}^m\lambda_k(1-\lambda_k)}},
\]
where $x_m$ is defined recursively by 
$x_m= T_{\lambda_m}x_{m-1}\equiv \left((1-\lambda_m)\Id + \lambda_m T\right)x_{m-1}$
for $\lambda_m\in (0,1)$ ($m=1,2, \dots$).  
The operators 
$T_{\lambda_m}$ all have the same set of fixed points (namely $\Fix T$), so these results are 
complementary to \cite{HerLukStu19a} where it was shown \cite[Theorem 3.5]{HerLukStu19a} 
that sequences of random variables on compact metric spaces  generated by Algorithm 
\ref{algo:RFI} with  {\em paracontractions} such as 
$T_{\lambda_m}$ above converge {\em almost surely} to a random variable in $\Fix T$, 
assuming that this is nonempty (see Proposition \ref{thm:Hermer} in this context). 
Necessary and sufficient conditions for linear 
convergence
of the {\em iterates} were also determined in a more limited setting in 
\cite[Theorems 3.11 and 3.15]{HerLukStu19a}.  
The results of \cite{Baillon1996,CominettiSotoVaisman,HerLukStu19a}, however, do not apply to 
inconsistent
stochastic feasibility considered here.

\subsection{Main Results}
\label{sec:mainres}

All of our main results concern Markov operators $\Pcal$ with
  update function $\Phi(x,i)=T_i(x)$ and transition kernel 
  $p$ given by \eqref{eq:trans kernel} for self mappings 
  $\mymap{T_i}{\Ecal}{\Ecal}$.
For any $\mu_0\in \mathscr{P}_2(\Ecal)$, we denote 
the distributions of the iterates of Algorithm \ref{algo:RFI} by 
$\mu_{k} =  \mu_0 \mathcal{P}^{k} = \mathcal{L}(X_{k})$, 
and we denote  $d_{W_2}\paren{\mu_k,\inv\mathcal{P}}\equiv 
\inf_{\pi'\in\inv\mathcal{P}}W_2\paren{\mu_{k},\, \pi'}$.

In most of our main results, it will be assumed that $\inv\mathcal{P}\neq\emptyset$.  
The existence theory is already well developed and is surveyed in Section \ref{sec:existence} 
below.  
The main convergence result for nonexpansive mappings follows from a fundamental result 
of Worm \cite[Theorem 7.3.13]{WormPhd2010}. 
\begin{thm}[convergence of Ces\`{a}ro average in
  $\Ecal$] \label{cor:cesaroConvergenceRn}
  Let   $\mymap{T_{i}}{\Ecal}{\Ecal}$ be nonexpansive 
  ($i \in I$) and assume
  $\inv\mathcal{P}\neq\emptyset$.  Let $\mu \in
  \mathscr{P}(\Ecal)$ and $\nu_{k} = \tfrac{1}{k}
  \sum_{j=1}^{k} \mu \mathcal{P}^{j}$, then this sequence converges in the 
  Prokhorov-L\`evy metric to 
  an invariant probability measure for $\mathcal{P}$, i.e.\ $\nu_{k}
  \to \pi^{\mu}$ where 
  \begin{equation}\label{eq:rep Cesaro limit}
    \pi^{\mu} = \int_{\Supp \mu} \pi^{x} \mu(\dd{x}),
  \end{equation}
  where for each $x \in \Supp \mu \subset \mathbb{R}^n$ there exists the limit of
  $(\nu_{k}^{x})$ and it is denoted by the invariant measure
  $\pi^{x}$.
\end{thm}
\noindent When the mappings are averaged, we obtain the 
following stronger result. It is worth pointing interested readers to an
analogous metric space result of \cite[Theorem 27]{BLL} in which 
it is shown that, on $p$-uniformly convex spaces, sequences 
generated by fixed point iterations of compositions of pointwise 
averaged mappings $T_i$  
converge in a weak sense whenever
$\cap_i \Fix T_i$ is nonempty.  When the composition is boundedly compact, then 
the fixed point iteration converges strongly to a fixed point.
\begin{thm}[convergence for averaged
  mappings on $\Ecal$]%
  \label{thm:a-firm convergence Rn} Let
  $\mymap{T_{i}}{\Ecal}{\Ecal}$ be
  averaged with constant 
  $\alpha_{i}\le \alpha<1$ ($i\in I$).
  Assume $\inv\mathcal{P}\neq\emptyset$. For any initial distribution $\mu_0 \in
  \mathscr{P}(\Ecal)$ the distributions $\mu_k$ of the iterates
   generated by Algorithm \ref{algo:RFI} converge in the 
  Prokhorov-L\`evy metric to an invariant probability measure for
  $\mathcal{P}$.
\end{thm}
The proof of this result is very different than the strategy applied to the analogous
metric space result of \cite[Theorem 27]{BLL}.

\section{Background Theory and Proofs}\label{sec:theory} 
In this section we prepare tools to prove the main results from
Section \ref{sec:mainres}. We start by establishing convergence results on the
supports of ergodic measures on a general Polish space $G$ (a complete separable 
metric space), and then,
for global convergence analysis of Algorithm \ref{algo:RFI}, we restrict ourselves to
$\Ecal$.  We begin with existence of invariant measures. 
We then analyze properties of (and convergence of the RFI on) subsets of 
$G$, called ergodic sets. Then we turn our attention to the
global convergence analysis.

\subsection{Existence of Invariant Measures}
\label{sec:existence}

A sequence of probability measures $(\nu_{k})$ 
is called \emph{tight}
if for any $\epsilon >0$ there exists a compact $K \subset G$ with
$\nu_{k} (K) > 1-\epsilon$ for all $k \in \mathbb{N}$.  By Prokhorov's theorem
(see, for instance, \cite{Billingsley}),  a sequence $(\nu_{k}) \subset
  \mathscr{P}(G)$, for $G$ a Polish space, 
  is tight if and only if
  $(\nu_{k})$ is compact in $\mathscr{P}(G)$, i.e.\ any
  subsequence of $(\nu_{k})$ has a  subsequence that converges in distribution
  (see, for instance \cite{Billingsley}). 
%
%

A basic building block is the existence of invariant measures proved by 
{Lasota} and T. {Szarek} \cite[Proposition 3.1]{LasSza06}.  Based on this, 
we show how existence  can be verified easily.  But first, we show how
to obtain existence constructively. 
\begin{prop}[construction of an invariant
  measure] \label{thm:construction_inv_meas} Let $\mu \in
  \mathscr{P}(G)$ and $\mathcal{P}$ be a Feller Markov operator.  Let
  $(\mu \mathcal{P}^{k})_{k \in \mathbb{N}}$ be a tight sequence of
  probability measures on a Polish space $G$, and let 
  $\nu_{k} = \tfrac{1}{k}\sum_{j=1}^{k} \mu\mathcal{P}^{j}$.  
  Any cluster point of the sequence $(\nu_{k})_{k\in\Nbb}$ is an invariant measure for
  $\mathcal{P}$.
\end{prop}
\begin{proof}
  Our proof follows \cite[Theorem 1.10]{Hairer2016}. 
  Tightness of the sequence
  $(\mu \mathcal{P}^{k})$ implies tightness of the sequence $(\nu_{k})$ and
  therefore  by Prokhorov's Theorem there exists a convergent subsequence $(\nu_{k_{j}})$ with
  limit $\pi \in \mathscr{P}(G)$. By the Feller
  property of $\mathcal{P}$ one has for any continuous and bounded
  $\mymap{f}{G}{\mathbb{R}}$ that also $\mathcal{P}f$ is continuous
  and bounded, and hence
  \begin{align*}
    \abs{(\pi \mathcal{P})f - \pi  f} &= \abs{\pi (\mathcal{P}f) - \pi
      f}  \\ &= \lim_{j} \abs{
      \nu_{k_{j}}(\mathcal{P} f)- \nu_{k_{j}} f} \\ &= \lim_{j}
    \frac{1}{k_{j}} \abs{ \mu\mathcal{P}^{k_{j}+1}f -
      \mu\mathcal{P}f} \\ &\le \lim_{j}
    \frac{2\norm{f}_{\infty}}{k_{j}} \\ &=0.
  \end{align*}
  Now, $\pi f = (\pi \mathcal{P})f$ for all $f \in C_{b}(G)$ implies
  that $\pi = \pi \mathcal{P}$.
\end{proof}

When a Feller Markov chain converges in distribution (i.e.\
$\mu\mathcal{P}^{k} \to \pi$), it does so to an invariant measure
(since $\mu\mathcal{P}^{k+1} \to \pi \mathcal{P}$).  A Markov operator
need not possess a unique invariant probability measure or any
invariant measure at all.  Indeed, for the
case that $T_{i}=P_{i}$, $i \in I$ is a projector onto a nonempty
closed and convex set $C_{i} \subset \Ecal$. A sufficient
condition for the deterministic Alternating Projections Method to
converge in the inconsistent case to a limit cycle for convex sets is
that one of the sets is compact 
(this is an easy consequence of \cite[Theorem 4]{CheneyGoldstein59}).  
Translating this into the present 
setting,  a sufficient condition for the existence of an invariant
measure for $\mathcal{P}$ is the existence of a compact set 
$K\subset \Ecal$ and  $\epsilon>0$ such that 
$p(x,K)\ge \epsilon$ 
for all $x \in\Ecal$. This holds, for instance, when 
there are only finitely many sets with one of them, say
$C_{\ibar}$, compact and $\mathbb{P}(\xi = \ibar) = \epsilon$, since
$p(x,C_{\ibar}) = \mathbb{P}(P_{\xi}x \in C_{\ibar}) \ge \mathbb{P}(P_{\xi}x \in 
C_{\ibar}, \xi= \ibar) = \mathbb{P}(\xi = \ibar) = \epsilon$ for all $ x \in
\Ecal$.  More generally, we have the following result.
\begin{prop}[existence of invariant measures for finite collections of
  continuous mappings]\label{cor:finite_selection_existence} Let
  $G$ be a Polish space and let $\mymap{T_{i}}{G}{G}$ be
  continuous for $i \in I$, where $I$ is a finite index set. If for one
  index $i \in I$ it holds that $\mathbb{P}(\xi = i)>0$ and $T_{i}(G)
  \subset K$, where $K \subset G$ is compact, then there exists an
  invariant measure for $\mathcal{P}$.
\end{prop}
\begin{proof}
  We have from $T_{i}(G)\subset K$ that $    \mathbb{P}(T_{\xi}x \in K ) \ge
  \mathbb{P}(\xi = i)$ and hence for the sequence $(X_{k})$ generated
  by Algorithm \ref{algo:RFI} for an arbitrary initial probability measure
  \begin{align*}
    \mathbb{P}(X_{k+1} \in K) = \mathbb{E}[\cpr{T_{\xi_{k}}X_{k}
      \in K}{X_{k}}] \ge \mathbb{P}(\xi = i) \qquad \forall k \in \mathbb{N}.
  \end{align*}
  The assertion follows now immediately from 
  \cite[Proposition 3.1]{LasSza06}
  since $\mathbb{P}(\xi = i)>0$ and
  $\mathcal{P}$ is Feller by continuity of $T_{j}$ for all $j \in I$.
\end{proof}

Next we mention an existence result which requires that the RFI sequence
$(X_{k})$ possess a uniformly bounded expectation.

\begin{prop}[existence in $\Ecal$,
  RFI] \label{thm:ex_inRn} Let
  $\mymap{T_{i}}{\Ecal}{\Ecal}$ ($i \in I$) be
  continuous. Let $(X_{k})$ be the RFI sequence (generated by
  Algorithm \ref{algo:RFI}) for some initial measure. Suppose that for all $k
  \in \mathbb{N}$ it holds that $\mathbb{E}\left[ \norm{X_{k}}\right]
  \le M$ for some $M \ge 0$.  Then there exists an invariant measure
  for the RFI Markov operator $\mathcal{P}$ given by \eqref{eq:trans
    kernel}.
\end{prop}
\begin{proof}
  For any $\epsilon>M$ Markov's inequality implies that
  \begin{align*}
    \mathbb{P}(\norm{X_{k}} \ge \epsilon) \le
    \frac{\mathbb{E}\left[\norm{X_{k}}\right]}{\epsilon} \le
    \frac{M}{\epsilon} <1
  \end{align*}
  Hence,
  \begin{align*}
    \limsup_{k \to \infty} \mathbb{P}(\norm{X_{k}} \le \epsilon) \ge
    \limsup_{k \to \infty} \mathbb{P}(\norm{X_{k}} < \epsilon) \ge
    1-\frac{M}{\epsilon} >0.
  \end{align*}
  Existence of an invariant measure then follows from \cite[Proposition 3.1]{LasSza06}
  since closed balls in
  $\Ecal$ with finite radius are compact, $\mathbb{P}(X_{k}
  \in \cdot) = \mu\mathcal{P}^{k}$ and continuity of $T_i$ yields the
  Feller property for $\mathcal{P}$.
\end{proof}

\subsection{Ergodic theory of general Markov Operators}
\label{sec:suppCvg}

Recall that an invariant probability measure $\pi$ of a Markov operator 
$\mathcal{P}$ is called
\emph{ergodic}, if any $p$\emph{-invariant set}, i.e.\ $A \in
\mathcal{B}(G)$ with $p(x,A)=1$ for all $x \in A$, has $\pi$-measure
$0$ or $1$.  
In this section we study the properties
of the RFI Markov chain when it is initialized by a distribution 
in the support of any ergodic
measure for $\mathcal{P}$. The convergence
properties for these points can be much stronger than the convergence
properties of Markov chains initialized by measures with support
outside the support of the ergodic measures.

The consistent stochastic feasibility problem was analyzed in
\cite{HerLukStu19a} without the need of the notion of convergence of
measures since, as shown in Proposition \ref{thm:Hermer}, for consistent stochastic feasibility 
convergence of sequences defined by \eqref{eq:X_RFI} is almost sure, if they converge at all.  
More general convergence of measures is more challenging
as the next example illustrates.
\begin{example}[nonexpansive mappings, negative
  result]\label{eg:convergence of average}
  For nonexpansive mappings in general, one cannot expect that the
  sequence $(\mathcal{L}(X_{k}))_{k\in\mathbb{N}}$ converges to an invariant
  probability measure.  Consider the nonexpansive operator $T:=
  T_{1}x:= -x$ on $\mathbb{R}$ and set, in the RFI setup, $\xi= 1$ and
  $I=\{1\}$.  Then $X_{2k} = x$ and $X_{2k+1} = -x$ for all $k \in
  \mathbb{N}$, if $X_{0}\sim \delta_{x}$. This implies for $x\neq 0$
  that $(\mathcal{L}(X_{k}))$ does not converge to the invariant 
  distribution $\pi_{x} = \tfrac{1}{2}(\delta_{x} + \delta_{-x})$
  (depending on $x$), since $\mathbb{P}(X_{2k} \in B) = \delta_{x}(B)$
  and $\mathbb{P}(X_{2k+1} \in B) = \delta_{-x}(B)$ for $B \in
  \mathcal{B}(\mathbb{R})$. Nevertheless the Ces\`{a}ro average
  $\nu_{k}:= \tfrac{1}{k} \sum_{j=1}^{k} \mathbb{P}^{X_{j}}$ converges
  to $\pi_{x}$.
\end{example}

As Example \ref{eg:convergence of average} shows, meaningful notions of
ergodic convergence are possible (in our case, convergence of the
Ces\`{a}ro average) even when convergence in distribution can not be
expected. We start by collecting several general results for Markov
chains on Polish spaces.  In the next section we restrict ourselves to 
equicontinuous and Feller Markov operators.

The following decomposition theorem on Polish spaces is key to our
development.  Two measures $\pi_{1},\pi_{2}$ are called {\em mutually
  singular} when there is $A \in \mathcal{B}(G)$ with $\pi_{1}(A^{c})
= \pi_{2}(A)=0$.  For more detail see, for instance, \cite{Walters82}.  
\begin{prop}%
\label{thm:decomp_ergodic_stat_measures}
  Denote by $\mathcal{I}$ the set of all invariant probability
  measures for $\mathcal{P}$ and by $\mathscr{E} \subset \mathcal{I}$
  the set of all those that are ergodic. Then, $\mathcal{I}$ is convex
  and $\mathscr{E}$ is precisely the set of its extremal
  points. Furthermore, for every invariant measure $\pi \in
  \mathcal{I}$, there exists a probability measure $q_{\pi}$ on
  $\mathscr{E}$ such that
  \begin{align*}
    \pi(A) = \int_{\mathscr{E}} \nu(A) q_{\pi}(\dd{\nu}).
  \end{align*}
  In other words, every invariant measure is a convex combination of
  ergodic invariant measures. Finally, any two distinct elements of
  $\mathscr{E}$ are mutually singular.
\end{prop}
\begin{rem}
  If there exists only one invariant probability measure of
  $\mathcal{P}$, we know by Proposition \ref{thm:decomp_ergodic_stat_measures}
  that it is ergodic. If there exist more invariant probability
  measures, then there exist uncountably many invariant and at least
  two ergodic probability measures.
\end{rem}

\begin{prop}\label{cor:Birkhoff_conditional}
  Let $\pi$ be an ergodic invariant probability measure for
  $\mathcal{P}$,  let $(G,\mathcal{G})$ be a measurable space, and let 
  $\mymap{f}{G}{\mathbb{R}}$ be measurable, bounded and satisfy $\pi \abs{f}^{p} < \infty$
  for $p \in [1,\infty]$.  
  Then
  \begin{align*}
    \nu_{k}^{x}f := \frac{1}{k} \sum_{j=1}^{k} p^{j}(x,f) \to \pi f
    \qquad \text{as } k \to \infty \quad\text{ for } \pi\text{-a.e. }
    x \in G,
  \end{align*}
  where $p^{j}(x,f) := \delta_{x}\mathcal{P}^{j}f = \cex{f(X_{j})}{X_{0}=x}$
  for the sequence $(X_{k})$ generated by Algorithm \ref{algo:RFI} with
  $X_{0} \sim \pi$.
\end{prop}
\begin{proof}
 This is a direct consequence of Birkhoff's ergodic theorem, \cite[Theorem 9.6]{kallenberg1997}.
\end{proof}

For fixed $x$ in Proposition \ref{cor:Birkhoff_conditional}, we want the
assertion to be true for all $f \in C_{b}(G)$. This issue is addressed 
in the next section by restricting our attention to equicontinuous Markov operators.
The results above do not require any explicit structure on the
mappings $T_i$ that generate the transition kernel $p$ and hence the
Markov operator $\mathcal{P}$, however the assumption that the initial 
random variable $X_0$ has the same distribution as the invariant measure $\pi$
is very strong.  For Markov operators generated from discontinuous
mappings $T_i$, the support of an invariant measure 
may not be invariant under
$T_{\xi}$.  To see this, let
\begin{align*}
  Tx :=
  \begin{cases}
    x, & x \in \mathbb{R}\setminus \mathbb{Q}\\
    -1, & x \in \mathbb{Q}.
  \end{cases}
\end{align*}
The transition kernel is then $p(x,A) = \1_{A}(T x)$ for $x \in
\mathbb{R}$ and $A \in \mathcal{B}(\mathbb{R})$. Let $\mu$ be the
uniform distribution on $[0,1]$, then, since $\lambda$-a.s.\ $T=\Id$
(where $\lambda$ is the Lebesgue measure on $\mathbb{R}$), we have
that $\mu \mathcal{P}^{k} = \mu$ for all $k \in
\mathbb{N}$. Consequently, $\pi = \mu$ is invariant and
$\Supp \pi=[0,1]$, but $T ([0,1]) = \{-1\} \cup [0,1]\cap
(\mathbb{R}\setminus \mathbb{Q})$, which is not contained in $[0,1]$.

The next result shows, however, that invariance of the the support of invariant 
measures under {\em continuous mappings} $T_i$ is guaranteed. 
\begin{lemma}[invariance of the support of invariant
  measures]\label{lemma:support_invariant_distr}
  Let $G$ be a Polish space and let $\mymap{T_{i}}{G}{G}$ be
  continuous for all $i \in I$. For any
  invariant probability measure $\pi \in \mathscr{P}(G)$ of
  $\mathcal{P}$ it holds that $T_{\xi} S_{\pi} \subset S_{\pi}$ a.s.
  where $S_{\pi} := \Supp \pi$.
\end{lemma}
\begin{proof}
By Fubini's Theorem,  for any $A \in \mathcal{B}(G)$  it holds that
  \begin{align*}
    \pi(A) = \int_{S_{\pi}} p(x,A) \pi(\dd{x}) &= \int_{\Omega}
    \int_{S_{\pi}} \1_{A}(T_{\xi}x) \pi(\dd{x}) \dd{\mathbb{P}} \\
    &= \int_{\Omega} \int_{S_{\pi}} \1_{T_{\xi(\omega)}^{-1} A}(x)
    \pi(\dd{x}) \mathbb{P}(\dd{\omega}) \\ &= \mathbb{E}\left[
      \pi(T_{\xi}^{-1}A \cap S_{\pi})\right] = \mathbb{E}\left[ 
\pi(T_{\xi}^{-1}A)\right].
  \end{align*}
  From $1= \pi(S_{\pi})= \mathbb{E}\left[ \pi(T_{\xi}^{-1}
    S_{\pi})\right]$ and $\pi(\cdot)\le 1$, it follows that
  $\pi(T_{\xi}^{-1}S_{\pi}) = 1$ a.s. 
  
  Note that $T_{i}^{-1}S_{\pi}$
  is closed for all $i \in I$ due to continuity of $T_{i}$ and
  closedness of $S_{\pi}$.  We show that $S_{\pi} \subset
  T_{\xi}^{-1}S_{\pi}$ a.s. which then yields the claim.  
  To see this, 
  let $S\subset G$ 
  be any closed set with $\pi(A\cap S) = \pi(A)$ for all $A\in\mathcal{B}(G)$, and
  let $x\in S_{\pi}$.  Then $\pi(\mathbb{B}(x,\epsilon) \cap S) >0$ 
  for all $\epsilon >0$, 
  i.e.\ $\mathbb{B}(x,\epsilon)\cap S \neq\emptyset$ for all $\epsilon>0$. 
  Now consider $x_{k} \in\mathbb{B}(x,\epsilon_{k}) \cap S$, 
  where $\epsilon_{k} \to 0$ as $k \to \infty$. Then since $S$ is closed, 
  $x_{k} \to x \in S$, from which we conclude that $S_{\pi}\subset S$.  
  Specifically, let $S=T^{-1}S_{\pi}$ and note that $T^{-1}S_{\pi} = D \setminus G$ 
  for some $D$ with
    $\pi(D)=0$. For any $A \in \mathcal{B}(G)$ it holds that $\pi(A \cap S) = 
    \pi(A) - \pi(A \cap D) = \pi(A)$. From the argument above, we conclude that 
    $S_{\pi}\subset S = T^{-1}S_{\pi}$ as claimed.  
\end{proof}

The above result means that, if the random variable $X_{k}$ enters $S_{\pi}$
for some $k$, then it will stay in $S_{\pi}$ forever. This can be
interpreted as a mode of convergence, i.e.\ convergence to the set
$S_{\pi}$, which is closed under application of $T_{\xi}$ a.s.
Equality $T_{\xi} S_{\pi} = S_{\pi}$ a.s.\ cannot be expected in
general.  For example,  let $I=\{1,2\}$, $G = \mathbb{R}$ and $T_{1}x= -1$,
$T_{2}x = 1$, $x\in\mathbb{R}$ and $\mathbb{P}(\xi = 1) = 0.5 =
\mathbb{P}(\xi = 2)$, then $\pi = \tfrac{1}{2}(\delta_{-1} +
\delta_{1})$ and $S_{\pi} = \{-1,1\}$. So $T_{1} S_{\pi} = \{-1\}$ and
$T_{2} S_{\pi} = \{1\}$.

\subsection{Ergodic convergence theory for equicontinuous Markov operators}
\label{sec:ergodicNonexp}

As shown by Szarek \cite{Szarek2006} and Worm \cite{WormPhd2010}, 
equicontinuity of Markov operators and
their generalizations give a nice structure to the set of ergodic
measures. We collect some results here which will be used heavily in the 
subsequent analysis.

\begin{definition}[equicontinuity]
  A Markov operator is called equicontinuous if
  $(\mathcal{P}^{k}f)_{k \in \mathbb{N}}$ is equicontinuous for all
  bounded and Lipschitz continuous $\mymap{f}{G}{\mathbb{R}}$.
\end{definition}

In the following we consider the union of supports of all ergodic
measures defined by
\begin{align}\label{d:S}
  S := \bigcup_{\pi \in \mathscr{E}} \Supp \pi,
\end{align}
where $\mathscr{E} \subset \inv \mathcal{P}$ denotes the set of
ergodic measures.
%

\begin{prop}[tightness of
  $(\delta_{s}\mathcal{P}^{k})$] \label{thm:tighnessOfIterates} Let
  $G$ be a Polish space.  Let $\mathcal{P}$ be equicontinuous.  Suppose there
  exists an invariant measure for $\mathcal{P}$. Then
the sequence  $(\delta_{s}\mathcal{P}^{k})_{k\in\Nbb}$ 
is tight for all $s \in S$ defined by
  \eqref{d:S}.
\end{prop}
\begin{proof}
  In the proof of \cite[Proposition 2.1]{Szarek2006} it is  shown
  that, under the assumption that $\mathcal{P}$ is equicontinuous and
  \begin{align}\label{eq:Kurtz}
    (\exists s,x\in G) \quad \limsup_{k \to \infty}
    \nu_{k}^{x}(\mathbb{B}(s,\epsilon)) > 0 \quad \forall\epsilon>0,
  \end{align}
  then the sequence $(\delta_{s}\mathcal{P}^{k})$ is tight.  
  It remains to demonstrate \eqref{eq:Kurtz}.  To see this, let $f =
  \1_{\mathbb{B}(s,\epsilon)}$ for some $s \in S_{\pi}$, where $\pi
  \in \mathscr{E}$ and $\epsilon>0$ in
  Proposition \ref{cor:Birkhoff_conditional}.  Then for $\pi\text{-a.e. } x \in
  G$ and $\nu_{k}^{x} := \tfrac{1}{k} \sum_{j=1}^{k} \delta_{x}
  \mathcal{P}^{j}$ we have
  \begin{align*}
    \limsup_{k \to \infty} \nu_{k}^{x}(\mathbb{B}(s,\epsilon)) =
    \lim_{k} \nu_{k}^{x}(\mathbb{B}(s,\epsilon)) =
    \pi(\mathbb{B}(s,\epsilon)) > 0.
  \end{align*}
This completes the proof.
\end{proof}
\begin{rem}[tightness of $(\nu_{k}^{s})$]\label{rem:tightnessMean}
  Note that the sequence $(\nu_{k}^{s})$ is tight for $s \in S$, since by
  Proposition \ref{thm:tighnessOfIterates}, for all $\epsilon>0$, there is a
  compact subset $K \subset G$ such that $p^{k}(s,K) > 1-\epsilon$ for
  all $k \in \mathbb{N}$, and hence also $\nu_{k}^{s}(K)>1-\epsilon$
  for all $k \in \mathbb{N}$.
\end{rem}

The next result due to Worm 
(Theorems 5.4.11 and 7.3.13 of \cite{WormPhd2010}))
concerns Ces\'{a}ro averages for
equicontinuous Markov operators.
\begin{prop}[convergence of Ces\`{a}ro averages 
{\cite{WormPhd2010}}]\label{thm:wormCesaro}
  Let $\mathcal{P}$ be Feller and equicontinuous, let $G$ be a Polish space
  and let $\mu \in \mathscr{P}(G)$.  Then the sequence 
  $(\nu_{k}^{\mu})$ is tight 
  ($\nu_{k}^{\mu} := \frac{1}{k} \sum_{j=1}^{k}\mu \mathcal{P}^{j}$)
  if and only if $(\nu_{k}^{\mu})$ converges to a 
  $\pi^{\mu} \in \inv \mathcal{P}$. In this case
  \begin{align*}
    \pi^{\mu} = \int_{\Supp \mu} \pi^{x} \mu(\dd{x}),
  \end{align*}
  where for each $x \in \Supp \mu \subset G$ there exists the limit of
  $(\nu_{n}^{x})$ and it is denoted by the invariant measure
  $\pi^{x}$.
\end{prop}

For the case the initial measure $\mu$ is supported in $\bigcup_{\pi
  \in \inv \mathcal{P}} \Supp {\pi}$, we have the following.
\begin{prop}[ergodic decomposition]\label{prop:ergodic_decomp}
  Let $G$ be a Polish space and let $\mathcal{P}$ be Feller and
  equicontinuous.  Then
  \begin{align*}
    S = \bigcup_{\pi \in \inv \mathcal{P}} \Supp {\pi},
  \end{align*}
  where $S$ is defined in \eqref{d:S}. Moreover $S$ is closed, and
  for any $\mu \in \mathscr{P}(S)$ it holds that $\nu_{k}^{\mu} \to
  \pi^{\mu}$ as $k\to\infty$ with
  \begin{align*}
    \pi^{\mu} = \int_{S} \pi^{x} \mu(\dd{x}),
  \end{align*}
  where $\pi^{x}$ is the unique ergodic measure with $x \in
  \Supp {\pi^{x}}$.
\end{prop}
\begin{proof}
  This is a consequence of \cite[Theorem 7.3.4]{WormPhd2010} and
  \cite[Theorem 5.4.11]{WormPhd2010}, Remark \ref{rem:tightnessMean} and
  Proposition \ref{thm:wormCesaro}.
\end{proof}
  Proposition \ref{prop:ergodic_decomp} only establishes convergence of the
  Markov chain when it is initialized with a measure in the support of
  an invariant measure; moreover, it is only the average of the
  distributions of the iterates that converges.
\begin{rem}[convergence of $(\nu_{k}^{s})$ on Polish
  spaces] \label{cor:weakCvgMetricSpace} Let $\pi$
  be an ergodic invariant probability measure for $\mathcal{P}$. Then
  for all $s \in \Supp \pi$ the sequence $\nu_{k}^{s} \to
  \pi$ as $k \to \infty$, where $\nu_{k}^{s} = \tfrac{1}{k}
  \sum_{j=1}^{k} p^{j}(s,\cdot)$.
\end{rem}

By Proposition \ref{thm:decomp_ergodic_stat_measures} any invariant measure can
be decomposed into a convex combination of ergodic invariant measures;
in particular two ergodic measures $\pi_{1},\pi_{2}$ are mutually
singular. Note that it still could be the case that $\Supp \pi_{1}
\cap \Supp \pi_{2} \neq \emptyset$. But Remark \ref{cor:weakCvgMetricSpace}
above establishes that for a Feller and equicontinuous \ Markov operator
$\mathcal{P}$ this is not possible, so the singularity of ergodic
measures extends to their support. This leads to the following 
corollary. 
\begin{cor}\label{th:singular measure char}
  Under the assumptions of Proposition \ref{prop:ergodic_decomp} for two ergodic
  measures $\pi, \tilde \pi$, the intersection 
  $\paren{\Supp {\pi}} \cap \paren{\Supp {\tilde \pi}} = \emptyset$ 
  if and only if $\pi \neq \tilde \pi$.
\end{cor}

\subsection{Ergodic theory for nonexpansive mappings}
\label{sec:nexp}

We now specialize to the case that the family of mappings
$\{T_{i}\}_{i \in I}$ are nonexpansive operators. 
The next technical lemma implies that every point in the support of an
ergodic measure is reached infinitely often starting from any other
point in this support.
\begin{lemma}[positive transition probability for ergodic
  measures] \label{lemma:pos_transitionKernel_ergodic} Let $G$ be
  a Polish space and let $\mymap{T_{i}}{G}{G}$ be nonexpansive, $i \in
  I$. Let $\pi$ be an ergodic invariant probability measure for
  $\mathcal{P}$. Then for any $s,\tilde s \in \Supp {\pi}$ it holds that
  \begin{align*}
    \forall \epsilon>0 \, \exists \delta>0, \,\exists (k_{j})_{j \in \mathbb{N}} 
\subset
    \mathbb{N} \,:\, p^{k_{j}}(s, \mathbb{B}(\tilde s, \epsilon)) \ge
    \delta \quad \forall j \in \mathbb{N}.
  \end{align*}
\end{lemma}
\begin{proof}
  Given $\tilde s \in \Supp {\pi}$ and $\epsilon>0$, find a continuous and
  bounded function $\mymap{f=f_{\tilde s, \epsilon} }{G}{ [0,1]}$ with
  the property that $f = 1$ on $\mathbb{B}(\tilde s,
  \tfrac{\epsilon}{2})$ and $f =0$ outside $\mathbb{B}(\tilde s,
  \epsilon)$. For $s \in \Supp {\pi}$ let $X_{0} \sim \delta_{s}$ and
  $(X_{k})$ generated by Algorithm \ref{algo:RFI}. By
  Remark \ref{cor:weakCvgMetricSpace} the sequence $(\nu_{k})$ converges to $\pi$
  as $k\to\infty$, where
  $\nu_{k}:=\tfrac{1}{k} \sum_{j=1}^{k} p^{j}(s,\cdot)$. So in
  particular $\nu_{k}f \to \pi f \ge \pi(\mathbb{B}(\tilde s,
  \tfrac{\epsilon}{2})) >0$ as $k \to \infty$. Hence, for $k$ large
  enough there is $\delta >0$ with
  \begin{align*}
    \nu_{k}f = \tfrac{1}{k} \sum_{j=1}^{k} p^{j}(s,f) \ge \delta.
  \end{align*}
  Now, we can extract a sequence $(k_{j}) \subset \mathbb{N}$ with
  $p^{k_{j}}(s,f) \ge \delta$, $j\in \mathbb{N}$ and hence
  \begin{align*}
    p^{k_{j}}(s,\mathbb{B}(\tilde s, \epsilon)) \ge p^{k_{j}}(s,f) \ge
    \delta >0.&\qedhere
  \end{align*}
\end{proof}

\begin{lemma}\label{lemma:e.c.}
  Let $G$ be a Polish space. Let $\mymap{T_{i}}{G}{G}$ be
  nonexpansive, $i \in I$ and let $\Pcal $ denote the
  Markov operator that is induced by the transition
  kernel in \eqref{eq:trans kernel}. 
  \begin{enumerate}[(i)]
  \item $\mathcal{P}$ is Feller.
  \item $\mathcal{P}$ is equicontinuous. 
  \end{enumerate}
\end{lemma}
\begin{proof}
  \begin{enumerate}[(i)]
  \item The mapping $T_{i}$ for $i \in I$ is 1-Lipschitz continuous,
    so in particular it is continuous. Proposition \ref{thm:Feller} yields the assertion.
  \item Let $\epsilon>0$ and $x,y \in G$ with $d(x,y)<
    \epsilon/\norm{f}_{\text{Lip}}$, then, using Jensen's inequality,
    Lipschitz continuity of $f$ and nonexpansivity of $T_{i}$, we get
    \begin{align*}
      \abs{\delta_{x}\mathcal{P}^{k}f-\delta_{y}\mathcal{P}^{k}f} &=
      \abs{\mathbb{E}[f(X_{k}^{x})] - \mathbb{E}[f(X_{k}^{y})]}\\ &\le
      \mathbb{E}[\abs{f(X_{k}^{x}) - f(X_{k}^{y})}] \\ &\le
      \norm{f}_{\text{Lip}} \mathbb{E}[d(X_{k}^{x},X_{k}^{y})] \\ &\le
      \norm{f}_{\text{Lip}} \mathbb{E}[d(x,y)] < \epsilon
    \end{align*}
    for all $k \in \mathbb{N}$. \qedhere
  \end{enumerate}
\end{proof}

A very helpful fact used later on is that the distance between the
supports of two ergodic measures is attained; moreover, any point in
the support of the one ergodic measure has a nearest neighbor in the
support of the other ergodic measure. 
\begin{lemma}[distance of supports is
  attained]\label{lemma:dist_supports_attained}
  Let $G$ be a Polish space and $\mymap{T_{i}}{G}{G}$ be
  nonexpansive, $i \in I$. Suppose $\pi,\tilde \pi$ are ergodic
  probability measures for $\mathcal{P}$. Denote the support of a 
  measure $\pi$ by $S_{\pi}\equiv \Supp {\pi}$.
  Then for all $s\in S_{\pi}$
  there exists $\tilde s \in S_{\tilde\pi}$ with $d(s,\tilde s) =
  \dist(s,S_{\tilde\pi}) = \dist(S_{\pi}, S_{\tilde\pi})$.
\end{lemma}
\begin{proof}
  First we show, that $\dist(S_{\pi}, S_{\tilde \pi}) = \dist(s,
  S_{\tilde \pi})$ for all $s \in S_{\pi}$. Therefore, recall the
  notation $X_{k}^{x} = T_{\xi_{k-1}}\cdots T_{\xi_{0}} x$ and note
  that by nonexpansivity of $T_{i}$, $i \in I$ and
  Lemma \ref{lemma:support_invariant_distr} it holds a.s.\ that
  \begin{align*}
    \dist(X_{k+1}^{x},S_{\pi}) \le \dist(X_{k+1}^{x},T_{\xi_{k}}
    S_{\pi}) = \inf_{s \in S_{\pi}} d(T_{\xi_{k}}X_{k}^{x},
    T_{\xi_{k}}s) \le \dist(X_{k}^{x},S_{\pi})
  \end{align*}
  for all $x \in G$, $\pi \in \inv \mathcal{P}$ and $k \in
  \mathbb{N}$.  Suppose now there would exist an $\hat s \in S_{\pi}$
  with $\dist(\hat s, S_{\tilde\pi}) < \dist(s, S_{\tilde\pi})$. Then
  by Lemma \ref{lemma:pos_transitionKernel_ergodic} for all $\epsilon >0$
  there is a $k \in \mathbb{N}$ with $\mathbb{P}( X_{k}^{\hat s} \in
  \mathbb{B}(s,\epsilon)) > 0$ and hence
  \begin{align*}
    \dist(s, S_{\tilde\pi}) \le d(s,X_{k}^{\hat s}) +
    \dist(X_{k}^{\hat s}, S_{\tilde\pi}) \le \epsilon + \dist(\hat s,
    S_{\tilde\pi})
  \end{align*}
  with positive probability for all $\epsilon >0$, which is a
  contradiction. So, it holds that $\dist(\hat s, S_{\tilde\pi}) = \dist(s,
  S_{\tilde\pi})$ for all $s, \hat s \in S_{\pi}$.

  For $s \in S_{\pi}$ let $(\tilde s_{m}) \subset S_{\tilde \pi}$ be a
  minimizing sequence for $\dist(s,S_{\tilde \pi})$, i.e.\ $\lim_{m}
  d(s,\tilde s_{m}) = \dist(s,S_{\tilde \pi})$.  Now define a
  probability measure $\gamma_{k}^{m}$ on $G\times G$ via
  \begin{align*}
    \gamma_{k}^{m} f := \mathbb{E}\left[ \frac{1}{k} \sum_{j=1}^{k}
      f(X_{j}^{s}, X_{j}^{\tilde s_{m}}) \right]
  \end{align*}
  for measurable $\mymap{f}{G\times G}{\mathbb{R}}$. Then
  $\gamma_{k}^{m} \in C(\nu_{k}^{s}, \nu_{k}^{\tilde s_{m}})$
  where $C(\nu_{k}^{s}, \nu_{k}^{\tilde s_{m}})$ is 
the set of all couplings for $\nu_{k}^{s}$ and $\nu_{k}^{\tilde s_{m}}$ 
(see \eqref{eq:couplingsDef}).
Also,  by
  Lemma \ref{lemma:weakCVG_productSpace} and Remark \ref{cor:weakCvgMetricSpace}
  the sequence $(\gamma_{k}^{m})_{k\in \Nbb}$ is tight for fixed 
  $m \in \mathbb{N}$ and
  there exists a cluster point $\gamma^{m} \in C(\pi, \tilde \pi)$. The
  sequence $(\gamma^{m}) \subset C(\pi, \tilde \pi)$ is again tight by
  Lemma \ref{lemma:weakCVG_productSpace}.  Thus for any cluster point
  $\gamma \in C(\pi,\tilde \pi)$ and the bounded and continuous
  function $(x,y) \mapsto f^{M}(x,y)=\min( M, d(x,y) )$ this yields
  \begin{align*}
    \gamma_{k}^{m}d = \gamma_{k}^{m} f^{M} \searrow \gamma^{m} f^{M}
    \qquad \text{as } k \to \infty
  \end{align*}
  for all $M \ge d(s,\tilde s_{m})$, $m\in \mathbb{N}$. Since by the
  Monotone Convergence Theorem $\gamma^{m}f^{M} \nearrow \gamma^{m}d$
  as $m \to \infty$, it follows that $\gamma^{m}f^{M} = \gamma^{m} d$ for
  all $M \ge d(s,\tilde s_{m})$. The same argument holds for $M \ge
  d(s,\tilde s_{1})$ and a subsequence $(\gamma^{m_{j}})$ with limit
  $\gamma$ such that $\gamma d = \gamma f^{M}$.  Hence,
  \begin{align*}
    \gamma d = \gamma f^{M} = \lim_{j} \gamma^{m_{j}} f^{M} = \lim_{j}
    \gamma^{m_{j}} d \le \lim_{j} d(s,\tilde s_{m_{j}}) =
    \dist(s,S_{\tilde \pi}).
  \end{align*}
  In particular for $\gamma$-a.e.\ $(x,y) \in S_{\pi} \times S_{\tilde
    \pi}$ it holds that $d(x,y)= \dist(S_{\pi}, S_{\tilde \pi})$, because
  $d(x,y)\ge \dist(S_{\pi}, S_{\tilde \pi})$ on $S_{\pi} \times
  S_{\tilde \pi}$. Taking the closure of these $(x,y)$ in $G\times G$,
  we see that for any $s \in S_{\pi}$ there is $\tilde s \in S_{\tilde
    \pi}$ with $d(s,\tilde s) = \dist(S_{\pi}, S_{\tilde \pi})$ by
  Lemma \ref{lemma:couplings}.
\end{proof}

\subsection{Convergence for nonexpansive mappings in $\Ecal$}
\label{sec:setCvgNE}
By Proposition \ref{thm:wormCesaro} tightness of a sequence of Ces\`{a}ro averages 
is equivalent to convergence of said sequence.  So our focus is on tightness in the 
Euclidean space setting.

\begin{lemma}[tightness of $(\mu\mathcal{P}^{k})$ in
  $\Ecal$] \label{lemma:tightnessOfNu_n} Let
  $\mymap{T_{i}}{\Ecal}{\Ecal}$ be nonexpansive 
  for all 
  $i \in I$, and let $\inv \mathcal{P} \neq \emptyset$ for the corresponding 
  Markov operator.  The sequence $(\mu\mathcal{P}^{k})_{k\in\Nbb}$ is tight for any $\mu \in
  \mathscr{P}(\Ecal)$.
\end{lemma}
\begin{proof}
  First, let $\mu=\delta_{x}$ for $x \in \Ecal$. We know that 
  the sequence  $(\delta_{s} \mathcal{P}^{k})$ is tight for $s \in S$ by
  Proposition \ref{thm:tighnessOfIterates}. So for $\epsilon>0$ there is a
  compact $K \subset \Ecal$ with $p^{k}(s,K)\ge 1-\epsilon$
  for all $k \in \mathbb{N}$. Recall the definition of $X_{k}^{x}$ in
  \eqref{eq:X_RFI}.  Since a.s.\ $\norm{X_{k}^{x}-X_{k}^{s}} \le \norm{x-s}$,
  we have that $p^{k}(x,\cb(K,\norm{x-s})) = \mathbb{P}( X_{k}^{x} \in
  \cb(K,\norm{x-s})) \ge p^{k}(s,K)\ge 1-\epsilon$ for all $k \in
  \mathbb{N}$. Hence $(\delta_{x} \mathcal{P}^{k})$ is tight.

  Now consider  the initial random variable $X_0\sim \mu$ for any 
  $\mu \in \mathscr{P}(\Ecal)$. For given $\epsilon >0$
  there is a compact $K_{\epsilon}^{\mu} \subset \Ecal$ with
  $\mu(K_{\epsilon}^{\mu}) > 1-\epsilon$. From the special case
  established above, there exists a compact $K_{\epsilon} \subset
  \Ecal$ with $p^{k}(0, K_{\epsilon}) >1-\epsilon$ for all $k
  \in \mathbb{N}$. Let $M>0$ such that $K_{\epsilon}^{\mu} \subset
  \cb(0,M)$ and let $x \in \cb(0,M)$. We have that
  $p^{k}(x,\cb(K_{\epsilon},M)) > 1-\epsilon$ for all $x \in
  \cb(0,M)$, since $\norm{X_{k}^{x} - X_{k}^{X_0}} \le \norm{x} \le
  M$. Hence $\mu\mathcal{P}^{k}(\cb(K_{\epsilon},M)) >
  (1-\epsilon)^{2}$, which implies tightness of the sequence
  $(\mu\mathcal{P}^{k})$.
\end{proof}
\begin{rem}[tightness of $(\nu_{k}^{\mu})$ in
  $\Ecal$] \label{rem:tightnessNu_nInRn} The tightness of
 the sequence $(\nu_{k}^{\mu})$ for any $\mu \in \mathscr{P}(\Ecal)$
  follows immediately from tightness of $(\mu\mathcal{P}^{k})$ as in
  Remark \ref{rem:tightnessMean}.
\end{rem}

We are now in a position to prove the first main result.\\
\textbf{ Proof of Theorem \ref{cor:cesaroConvergenceRn}}.  
 By Lemma \ref{lemma:e.c.} the Markov operator $\mathcal{P}$ is Feller and 
 equicontinuous.  By Lemma \ref{lemma:tightnessOfNu_n}  the sequence 
 $\left(\mu\mathcal{P}^k\right)$ is tight, and so the sequence of 
 of Ces\'aro averages $(\nu_k^\mu)$ is also tight 
 (see Remark \ref{rem:tightnessNu_nInRn}).  
Hence by Proposition \ref{thm:wormCesaro} $\nu_k^\mu\to \pi^\mu$
with $\pi^\mu$ given by \eqref{eq:rep Cesaro limit}.\hfill$\Box$

\subsection{More properties of the RFI for nonexpansive mappings}
\label{sec:furtherprops}

This section is devoted to the preparation of some tools used in
Section \ref{sec:cvgAverages} to prove convergence of the distributions of
the iterates of the RFI. When the Markov chain is initialized with a
point not supported in $S$, i.e.\ when $\Supp \mu \setminus S \neq
\emptyset$, the convergence results on general Polish spaces are much
weaker than for the ergodic case in the previous section. One 
problem is that the sequences $(\nu_{k}^{x})_{k \in \mathbb{N}}$ for 
$x \in G \setminus S$ need not be tight anymore.  The right-shift operator 
$\mathcal{R}$ on
$l^{2}$, for example, with the initial distribution $\delta_{e_{1}}$,
generates the sequence $\mathcal{R}^{k} e_{1} = e_{k}$, $k=1,2,\dots$.  
Examples of spaces on which
we can always guarantee tightness are, of course, 
Euclidean spaces as seen in the
previous section, 
and compact metric spaces --
since then $(\mathscr{P}(G), d_{P})$ is compact.

For the case that the sequence of Ces\`{a}ro
averages does not necessarily converge, we have the following result.
\begin{lemma}[convergence of with  nonexpansive
  mappings] \label{thm:invMeas_nonexpans_map} Let $(G,d)$ be a 
  separable complete metric space and 
  let $\mymap{T_{i}}{G}{G}$ be nonexpansive for all $i \in I$. 
  Suppose $\inv \mathcal{P} \neq \emptyset$.  
  Let $X_{0} \sim \mu \in
  \mathscr{P}(G)$ and let $(X_{k})$ be the sequence generated by
  Algorithm \ref{algo:RFI}.  Denote the support of any measure $\mu$ by 
  $S_\mu$, and denote $\nu_{k} := \tfrac{1}{k}\sum_{j=1}^{k} \mu \mathcal{P}^{j}$.
  \begin{enumerate}[(i)]
  \item\label{item:nonExpProp_1}  
    \[\forall \pi \in \inv \mathcal{P}, \quad \dist(X_{k+1},S_{\pi}) \le \dist(X_{k},S_{\pi})~ a.s.
    \quad \forall k\in \mathbb{N}.\]
  \item\label{item:nonExpProp_2top} If the sequence $(\nu_{k})$ has a cluster point $\pi\in \inv\mathcal{P}$, then,
  \begin{enumerate}[(a)]
  \item\label{item:nonExpProp_2}$\dist(X_{k},S_{\pi}) \to 0$ a.s.\ as
    $k\to\infty $;
  \item\label{item:nonExpProp_3} all cluster points of the sequence $(\nu_{k})$ have the
    same support;
  \item \label{item:nonExpProp_4} cluster points of the sequence $(\mu
    \mathcal{P}^{k})$ have support in $S_{\pi}$ (if they exist).
  \end{enumerate}
\end{enumerate}
\end{lemma}

\begin{proof}
\eqref{item:nonExpProp_1}. By Lemma \ref{lemma:support_invariant_distr}, the sets 
  on which $T_{\xi_{k}} S_{\pi}$ is not a subset of
    $S_{\pi}$ are $\mathbb{P}$-null sets and their union
    is also a $\mathbb{P}$-null set.
    This yields
    \begin{align*}
(\forall s\in S_{\pi})\quad       
\dist(X_{k+1}, S_{\pi}) \le d(X_{k+1}, T_{\xi_{k}} s) =
      d(T_{\xi_{k}} X_{k},T_{\xi_{k}}s) \le d(X_{k}, s)
            \quad\text{a.s.},
    \end{align*}
    and hence
    \begin{align*}
      \dist(X_{k+1},S_{\pi}) \le \dist(X_{k},S_{\pi}) 
      \quad\text{a.s.}
    \end{align*}
    
\eqref{item:nonExpProp_2} Define the function $f = \min(M,\dist(\cdot,S_{\pi}))$ 
    for some $M>0$.  Since  this is bounded and continuous, we have for a subsequence
    $(\nu_{k_{j}})$  converging to $\pi$, that
    $\nu_{k_{j}} f = \tfrac{1}{k_{j}} \sum_{n=1}^{k_{j}}
    \mu\mathcal{P}^{n}f \to \pi f = 0$ as $j \to \infty$. Now
    part \eqref{item:nonExpProp_1} and the identity
    \[ \mu\mathcal{P}^{n+1}f =
    \mathbb{E}[\min(M,\dist(X_{n+1},S_{\pi}))] \le
    \mathbb{E}[\min(M,\dist(X_{n},S_{\pi}))] = \mu\mathcal{P}^{n}f
    \]
    yield $\mu\mathcal{P}^{n}f =
    \mathbb{E}[\min(M,\dist(X_{n},S_{\pi}))] \to 0$ as $n \to \infty$.
    Again by part \eqref{item:nonExpProp_1}
    \[
    Y:=\lim_{n\to\infty} \min(M,\dist(X_{n},S_{\pi}))
    \]
    exists and is nonnegative; so by Lebesgue's dominated convergence
    theorem it follows that $Y=0$ a.s., since otherwise $\mathbb{E}[Y]
    > 0 = \lim_{n\to\infty} \mu \mathcal{P}^{n} f$ would yield a
    contradiction.

\eqref{item:nonExpProp_3} Let $\pi_{1},\pi_{2}$ be two cluster points of $(\nu_{k})$ with
    support $S_{1}, S_{2}$ respectively, then these probability
    measures are invariant for $\mathcal{P}$ by
    Proposition \ref{thm:construction_inv_meas}. By Corollary \ref{th:singular measure
      char} the intersection $S_{1}\cap S_{2}$ must be nonempty. 
      Suppose now w.l.o.g. $\exists y
    \in S_{1}\setminus S_{2}$.  Then there is an $\epsilon>0$ with
    $\mathbb{B}(y,2\epsilon) \cap S_{2} = \emptyset$. Let
    $\mymap{f}{G}{[0,1]}$ be a continuous function that takes the
    value $1$ on $\mathbb{B}(y,\tfrac{\epsilon}{2})$ and $0$ outside
    of $\mathbb{B}(y,\epsilon)$. Then $\pi_{1}f >0$ and
    $\pi_{2}f=0$. But there are two subsequences of $(\nu_{k})$ with
    $\nu_{k_{j}} f \to \pi_{1}f$ and $\nu_{\tilde k_{j}}f \to
    \pi_{2}f$ as $j \to \infty$. For the former sequence we have, for
    $j$ large enough,
    \begin{align*}
      \exists \delta >0: \frac{1}{k_{j}} \sum_{n=1}^{k_{j}} \mu
      \mathcal{P}^{n} f \ge \delta >0.
    \end{align*}
    So, one can from this extract a sequence $(m_{k})_{k\in\Nbb} \subset
    \mathbb{N}$ with $\mu \mathcal{P}^{m_{k}} f \ge \delta$, $k\in
    \mathbb{N}$.  Note that $\mathbb{P}(X_{m_{k}} \in
    \mathbb{B}(y,\epsilon)) \ge \mu \mathcal{P}^{m_{k}}f \ge \delta
    >0$. This implies $\dist(X_{m_{k}},S_{2}) \ge \epsilon$ with
    $\mathbb{P}\ge \delta$ and hence $\mathbb{E}[\dist(X_{m_{k}},
    S_{2})] \ge \delta\epsilon$, in contradiction to
    \eqref{item:nonExpProp_2}. So there cannot be such $y$ which
    yields $S_{1} = S_{2}$, as claimed.

\eqref{item:nonExpProp_4} 
    Let $\nu$ be a cluster point of the sequence $(\mu\mathcal{P}^{k})$, which
    is assumed to exist, and assume there is $s \in \Supp \nu \setminus
    S_{\pi}$ and $\epsilon>0$ such that $\dist(s,S_{\pi})> 2
    \epsilon$. Let $\mymap{f}{G}{[0,1]}$ be a continuous function,
    that takes the value $1$ on $\mathbb{B}(s,\tfrac{\epsilon}{2})$
    and $0$ outside of $\mathbb{B}(s,\epsilon)$. With
    \eqref{item:nonExpProp_2} we find, that
    \begin{align*}
      0<\nu f = \lim_{j} \mathbb{P}^{X_{k_{j}}}f \le \lim_{j}
      \mathbb{P}(X_{k_{j}} \in \mathbb{B}(s,\epsilon)) = 0.
    \end{align*}
    Were $\mathbb{P}(X_{k_{j}} \in \mathbb{B}(s,\epsilon)) \ge
    \delta>0$ for $j$ large enough, then this would imply that
    \[
    \mathbb{E}[\dist(X_{k_{j}}, S_{\pi})] \ge \delta \epsilon,
    \]
    which is a contradiction.  We conclude that
    there is no such $s$, which completes the proof. \qedhere
\end{proof}

We now prepare some tools to handle convergence of the distributions
of the iterates of the RFI for averaged mappings in
Section \ref{sec:cvgAverages}. 
We restrict ourselves to Polish spaces with {\em finite dimensional metric}
(see Definition \ref{def:finiteDimMetric}) in order to apply a differentiation
theorem.  We begin with the next technical fact. 
\begin{lemma}[characterization of balls in
  $(\mathscr{E},d_{P})$]\label{lem:char_ballsEW}
  Let $(G,d)$ be a  separable complete space and 
  $\mymap{T_{i}}{G}{G}$ be nonexpansive for all 
  $i \in I$. Let $\mathscr{E}$ denote the (convex) set of ergodic
measures associated to the Markov operator $\Pcal$, which is induced 
by the family of mappings $\{T_i\}_{i\in I}$ and the marginal probability 
law of the random variables $\xi_k$.
  Let $\pi, \tilde\pi \in \mathscr{E}$ and denote the support of 
  the measure $\pi$ by $S_\pi$ (and similarly for $\tilde\pi$). Then
  \begin{align*}
    \tilde \pi \in \cb(\pi,\epsilon) \qquad \Longleftrightarrow \qquad
    S_{\tilde \pi} \subset \cb(S_{\pi},\epsilon)
  \end{align*}
  for $\epsilon \in (0,1)$, where $\cb(\pi,\epsilon)$ is the closed $\epsilon$-ball 
  with respect to the Prokhorov-L\`evy metric $d_{P}$.
\end{lemma}
\begin{proof}
  By Lemma \ref{lemma:dist_supports_attained} there exist $s \in S_{\pi}$
  and $\tilde s \in S_{\tilde \pi}$ such that $d(s,\tilde s ) =
  \dist(S_{\pi}, S_{\tilde \pi})$. First note that, if $\pi\neq \tilde
  \pi$, then $S_{\pi} \cap S_{\tilde\pi} = \emptyset$ by
  Corollary \ref{th:singular measure char}, and hence $d(s,\tilde s )
  = \dist(S_{\pi}, S_{\tilde \pi}) > 0$.\\
  Recall the notation $X_{k}^{x} \equiv T_{\xi_{k-1}} \cdots T_{\xi_{0}}
  x$ for $x \in G$ and note that by
  Lemma \ref{lemma:couplings}\eqref{item:couplings1} and
  Lemma \ref{lemma:support_invariant_distr}, $\Supp \mathcal{L}(X_{k}^{s})
  \subset S_{\pi}$ and $\Supp \mathcal{L}(X_{k}^{\tilde s}) \subset
  S_{\tilde \pi}$. So it holds that $d(X_{k}^{s},X_{k}^{\tilde s}) \ge
  \dist(S_{\pi},S_{\tilde \pi})$ a.s.\ for all $k\in \mathbb{N}$.
  Since $T_{i}$ ($i \in I$) is nonexpansive  we have that
  $d(X_{k}^{s},X_{k}^{\tilde s}) \leq d(s,\tilde s)$ a.s.\ for all $k
  \in \mathbb{N}$. So, both inequalities together imply the equality
  \begin{align}\label{eq:const_dist1}
    d(X_{k}^{s},X_{k}^{\tilde s}) = d(s,\tilde s) \qquad \text{a.s. }
    \forall k \in \mathbb{N}.
  \end{align}
  Now, letting $c:= \min(1,d(s,\tilde s))$, we show that
  $d_{P}(\pi,\tilde \pi) = c$, where $d_{P}$ denotes the
  Prokhorov-L\`evy metric (see
  Lemma \ref{lemma:prokhorovDist_properties}). Indeed,  take $(X,Y) \in
  C(\mathcal{L}(X_{k}^{s}), \mathcal{L}(X_{k}^{\tilde s}))$. Again, by
  Lemma \ref{lemma:couplings}\eqref{item:couplings1} and
  Lemma \ref{lemma:support_invariant_distr} $\Supp \mathcal{L}(X) \subset
  S_{\pi}$ and $\Supp \mathcal{L}(Y) \subset S_{\tilde \pi}$ and hence
  $d(X,Y) \ge \dist(S_{\pi}, S_{\tilde \pi}) = d(s,\tilde s)$ a.s. We
  have, thus
  \begin{align*}
    \mathbb{P}(d(X,Y)> c - \delta) \ge 
    \mathbb{P}(d(X,Y) > d(s,\tilde{s}) - \delta) = 1 \qquad \forall \delta >0,
  \end{align*}
  which implies $d_{P}(\mathcal{L}(X_{k}^{s}),
  \mathcal{L}(X_{k}^{\tilde s})) \ge c$ by
  Lemma \ref{lemma:prokhorovDist_properties}\eqref{item:prokLeviRep}. In
  particular, for $c=1$ it follows that $d_{P}(\mathcal{L}(X_{k}^{s}),
  \mathcal{L}(X_{k}^{\tilde s})) =1$, since $d_{P}$ is bounded by
  $1$. Now, let $c<1$, i.e.\ $c=d(s,\tilde s) < 1$. We have by
  \eqref{eq:const_dist1}
  \begin{align*}
    \inf_{(X,Y) \in C(\mathcal{L}(X_{k}^{s}),
      \mathcal{L}(X_{k}^{\tilde s}))} \mathbb{P}\left(d(X,Y) >
      c\right) \le \mathbb{P}\left(d(X_{k}^{s}, X_{k}^{\tilde s}) >
      c\right) = 0 \le c.
  \end{align*}
  Altogether we find that $d_{P}(\mathcal{L}(X_{k}^{s}),
  \mathcal{L}(X_{k}^{\tilde s})) = c$, again by
  Lemma \ref{lemma:prokhorovDist_properties}\eqref{item:prokLeviRep}. Since
  also $\Supp \nu_{k}^{s} \subset S_{\pi}$ and $\Supp \nu_{k}^{\tilde
    s} \subset S_{\tilde \pi}$, where $\nu_{k}^{x} = \tfrac{1}{k}
  \sum_{j=1}^{k}\mathcal{L}(X_{j}^{x})$ for any $x\in G$, it follows
  that
  \begin{align}\label{eq:charBalls}
    c \le d_{P}(\nu_{k}^{s},\nu_{k}^{\tilde s}) \le \max_{j=1,\ldots,k}
    d_{P}(\mathcal{L}(X_{j}^{s}), \mathcal{L}(X_{j}^{\tilde s})) =c
  \end{align}
  by Lemma \ref{lemma:prokhorovDist_properties}\eqref{item:prokhorov5}.
  Now taking the limit $k \to \infty$ of \eqref{eq:charBalls} and using
  Remark \ref{cor:weakCvgMetricSpace}, it follows that $d_{P}(\pi,\tilde\pi) = c$.
This proves the assertion.
\end{proof}

\begin{definition}[Besicovitch family]
  A family $\mathcal{B}$ of closed balls $B = \cb(x_{B},\epsilon_{B})$
  with $x_{B} \in G$ and $\epsilon_{B} > 0$ on the metric space
  $(G,d)$ is called a \emph{Besicovitch family} of balls if
  \begin{enumerate}[(i)]
  \item for every $B \in \mathcal{B}$ one has $x_{B} \not \in B' \in
    \mathcal{B}$ for all $B' \neq B$, and
  \item $\bigcap_{B \in \mathcal{B}} B \neq \emptyset$.
  \end{enumerate}
\end{definition}

\begin{definition}[$\sigma$-finite dimensional 
metric]\label{def:finiteDimMetric}
  Let $(G, d)$ be a metric space. We say that $d$ is \emph{finite
    dimensional} on a subset $D \subset G$ if there exist constants $K
  \ge 1$ and $0 < r \le \infty $ such that $\Card \mathcal{B} \le K$
  for every Besicovitch family $\mathcal{B}$ of balls in $(G, d)$
  centered on $D$ with radius $< r$.  We say that $d$ is
  \emph{$\sigma$-finite dimensional} if $G$ can be written as a
  countable union of subsets on which $d$ is finite dimensional.
\end{definition}

\begin{prop}[differentiation theorem, \cite{Preiss81}]
\label{th:diff thm}
  Let $(G, d)$ be a  separable complete metric space. For every locally
  finite Borel regular measure $\lambda$ over $(G, d)$, it holds that
  \begin{equation}\label{eq:diff thm}
    \lim_{r \to 0} \frac{1}{\lambda(\cb(x,r))} \int_{\cb(x,r)} f(y)
    \lambda(\dd{y}) = f(x) \qquad \text{for } \lambda\text{-a.e. } x
    \in G,~\forall f \in L_{\mathrm{loc}}^{1}(G,\lambda)
  \end{equation}
  if and only if $d$
  is $\sigma$-finite dimensional.
\end{prop}

\begin{prop}[Besicovitch covering property in 
$\mathscr{E}$]\label{prop:weak_besicovitch_Rn}
  Let $(G,d)$ be separable complete metric space with finite dimensional metric 
$d$ and
  let $\mymap{T_{i}}{G}{G}$ be nonexpansive, $i \in I$. The
  cardinality of any Besicovitch family of balls in
  $(\mathscr{E},d_{P})$ is bounded by the same constant that bounds
  the cardinality of Besicovitch families in $G$.
\end{prop}
\begin{proof}
  Let $\mathcal{B}$ be a Besicovitch family of closed balls $B=
  \cb(\pi_{B},\epsilon_{B})$ in $(\mathscr{E},d_{P})$, where $\pi_{B}
  \in \mathscr{E}$ and $\epsilon_{B} > 0$. Note that if $\epsilon_{B}
  \ge 1$, then $\abs{\mathcal{B}} = 1$, since in that case $B=
  \mathscr{E}$ since $d_{P}$ is bounded by $1$. So let
  $\abs{\mathcal{B}} >1$, that implies $\epsilon_{B}<1$ for all $B \in
  \mathcal{B}$.

  The defining properties of a Besicovitch family translate then with
  help of Lemma \ref{lem:char_ballsEW} into
  \begin{align} \label{eq:cvg_average_besicovitch1} \pi_{B} \not \in
    B', \quad \forall B' \in \mathcal{B}\setminus \{B\}
    &&\Longleftrightarrow&& S_{\pi_{B}} \cap \cb(
    S_{\pi_{B'}},\epsilon_{B'}) = \emptyset, \qquad \forall B' \in
    \mathcal{B}\setminus \{B\}, 
  \end{align}
  and
  \begin{align}
    \label{eq:cvg_average_besicovitch2}
    \bigcap_{B \in \mathcal{B}} B \neq \emptyset
    &&\Longleftrightarrow&& \bigcap_{B \in \mathcal{B}}
    \cb(S_{\pi_{B}},\epsilon_{B}) \neq \emptyset.
  \end{align}
  Now fix $\pi$ in the latter intersection in
  \eqref{eq:cvg_average_besicovitch2} and let $s \in S_{\pi}$. Also
  fix for each $B \in \mathcal{B}$ a point $s_{B} \in S_{\pi_{B}}$
  with the property that $s_{B} \in \argmin_{\tilde s \in S_{\pi_{B}}}
  d(s,\tilde s)$ (possible by Lemma \ref{lemma:dist_supports_attained}).
  Then the family $\mathcal{C}$ of balls $\cb(s_{B},\epsilon_{B})
  \subset G$, $B \in \mathcal{B}$ is also a Besicovitch family: We
  have $s_{B} \not \in B'$ for $B \neq B'$ due to
  \eqref{eq:cvg_average_besicovitch1} and by the choice of $s_{B}$ one
  has $s \in \bigcap_{B \in \mathcal{C}} B$.
  Since the cardinality of any Besicovitch family in $G$ is bounded by
  a uniform constant, it follows, that also the cardinality of
  $\mathcal{B}$ is uniformly bounded.
\end{proof}
\begin{rem}\label{rem:norm_fin_dim} The cardinality of any
  Besicovitch family in $\Ecal$ is uniformly bounded
  depending on $\dim\Ecal$ \cite[Lemma 2.6]{mattila1995geometry}.
\end{rem}

\begin{lemma}[equality around the support of ergodic measures implies
  equality of measures]\label{lemma:equalityBallsimpliesEqOfMeasures}
  Let $(G,d)$ be a separable complete metric space with the finite dimensional 
  metric $d$ and let $\mymap{T_{i}}{G}{G}$ be nonexpansive ($i \in I$).  If
  $\pi_{1},\pi_{2} \in \inv \mathcal{P}$ satisfy
  \begin{align}\label{eq:equalityBallsimpliesEqOfMeasures}
    \pi_{1}(\cb(S_{\pi}, \epsilon)) = \pi_{2}(\cb(S_{\pi}, \epsilon))
  \end{align}
  for all $\epsilon>0$ and all $\pi \in \mathscr{E}$, then $\pi_{1} =
  \pi_{2}$.
\end{lemma}
\begin{proof}
  From Proposition \ref{thm:decomp_ergodic_stat_measures} follows the existence
  of probability measures $q_{1},q_{2}$ on the set $\mathscr{E}$ of
  ergodic measures for $\mathcal{P}$ such that one has
  \begin{align*}
    \pi_{j}(A) = \int_{\mathscr{E}} \pi(A) q_{j}(\dd{\pi}), \qquad A
    \in \mathcal{B}(G),\, j=1,2.
  \end{align*}
  If we set $q = \tfrac{1}{2}(q_{1}+q_{2})$, then by the Radon-Nikodym 
  theorem, there are densities $f_{1},f_{2} \ge 0$ on $\mathscr{E}$
  with $q_{j} = f_{j} \cdot q$ and hence
  \begin{align*}
    \pi_{j}(A) = \int_{\mathscr{E}} \pi(A) f_{j}(\pi) q(\dd{\pi}),
    \qquad A \in \mathcal{B}(G),\, j=1,2.
  \end{align*}
  For $q$-measurable\ subsets $E \subset \mathscr{E}$, one can define a
  probability measure on $\mathscr{E}$ via
  \begin{align}\label{eq:pi char}
    \tilde \pi_{j} (E) := \int_{\mathscr{E}} \1_{E}(\pi) f_{j}(\pi)
    q(\dd{\pi}), \qquad j=1,2.
  \end{align}
  One then has for $\epsilon >0$ and $\pi \in \mathscr{E}$ that
  \begin{align}
    \label{eq:cvg_average_pi_equals_piTilde}
    \pi_{j}(\cb(S_{\pi},\epsilon)) = \tilde
    \pi_{j}(\cb(\pi,\epsilon)), \qquad j=1,2,
  \end{align}
  where $\cb(\pi,\epsilon) := \mysetc{\tilde \pi \in \mathscr{E}}{
    d_{P}(\tilde\pi,\pi) \le\epsilon}$. This is due to
  Lemma \ref{lem:char_ballsEW}, from which follows
  \begin{align*}
    \tilde\pi(\cb(S_{\pi},\epsilon)) =
    \begin{cases}
      1, & \tilde \pi \in \cb(\pi,\epsilon)\\
      0, & \text{else}
    \end{cases}.
  \end{align*}
  With the above characterizations of $\pi_j$ and $\tilde \pi_j$, we
  can use Proposition \ref{th:diff thm} to show that $f_{1}=f_{2}$ $q$-a.s.,
  which, together with \eqref{eq:pi char}, would imply that $\pi_{1} =
  \pi_{2}$, as claimed.  To apply Proposition \ref{th:diff thm} we require that
  $d_{P}$ is finite dimensional.  But this follows from
  Proposition \ref{prop:weak_besicovitch_Rn}.  So Proposition \ref{th:diff thm} applied to
  $\tilde \pi_{j}$ with respect to $q$ then gives $q$-a.s.
  \begin{equation}\label{eq:rapsberry}
    \lim_{\epsilon \to 0} \frac{\tilde \pi_{j}(\cb(\pi,
      \epsilon))}{q(\cb(\pi,\epsilon))} = f_{j}(\pi), \quad j=1,2.
  \end{equation}
  And since $\tilde\pi_{1}(\cb(\pi,\epsilon)) =
  \tilde\pi_{2}(\cb(\pi,\epsilon))$ by
  \eqref{eq:cvg_average_pi_equals_piTilde} and assumption
  \eqref{eq:equalityBallsimpliesEqOfMeasures}, we have $f_{1}=f_{2}$
  $q$-a.s., which completes the proof.
\end{proof}
\begin{rem}\label{rem:equality}
  In the assertion of Lemma \ref{lemma:equalityBallsimpliesEqOfMeasures},
  it is enough to claim the existence of a sequence
  $(\epsilon_{k}^{\pi})_{k \in \mathbb{N}} \subset \mathbb{R}_{+}$
  with $\epsilon_{k}^{\pi} \to 0$ as $k \to \infty$ satisfying
  \begin{align*}
    \pi_{1}(\cb(S_{\pi},\epsilon_{k}^{\pi})) =
    \pi_{2}(\cb(S_{\pi},\epsilon_{k}^{\pi})) \qquad \forall \pi\in
    \mathscr{E}, \, \forall k \in \mathbb{N},
  \end{align*}
  because from Proposition \ref{th:diff thm} one has the existence of the limit
  in \eqref{eq:rapsberry} $q$-a.s.
\end{rem}

\subsection{Convergence theory for averaged 
mappings}
\label{sec:cvgAverages}
  
Continuing the development of the convergence theory under
greater regularity assumptions on the mappings $T_{i}$ ($i \in I$),
in this section we examine what is achievable under the assumption
that the mappings $T_i$ are averaged (Definition \ref{d:a-fne}). 
We restrict ourselves to the Euclidean space
$\Ecal$, and begin with a technical lemma that describes properties of sequences
whose relative expected distances are invariant under $T_{\xi}$.
\begin{lemma}[constant expected separation]\label{lemma:const_dist}
  Let $\mymap{T_{i}}{\Ecal}{\Ecal}$ be averaged
with $\alpha_{i}\le \alpha<1$, $i\in I$. Let
  $\mu, \nu \in \mathscr{P}(\Ecal)$ and $X \sim \mu$, $Y \sim
  \nu$ independent of $(\xi_{k})$ satisfy
  \begin{align*}
    \mathbb{E}\left[\norm{X_{k}^{X} - X_{k}^{Y}}^{2}\right] =
    \mathbb{E}\left[\norm{X-Y}^{2}\right] \qquad \forall k\in
    \mathbb{N},
  \end{align*}
  where $X_{k}^{x}:= T_{\xi_{k-1}}\!\!\!\cdots T_{\xi_{0}}x$ for $x
  \in \Ecal$ is the RFI sequence started at $x$. Then for
  $\mathbb{P}^{(X,Y)}$-a.e.\ $(x,y) \in \Ecal\times
  \Ecal$ we have $X_{k}^{x}-X_{k}^{y}= x-y$
  $\mathbb{P}$-a.s.\ for all $k \in \mathbb{N}$. Moreover, if there
  exists an invariant measure for $\mathcal{P}$, then
  \begin{align*}
    \pi^{x}(\cdot) = \pi^{y}(\cdot-(x-y)) \qquad
    \mathbb{P}^{(X,Y)}\mbox{-a.s.}
  \end{align*}
  for the limiting invariant measures $\pi^{x}$ of the Ces\`{a}ro
  average of $(\delta_{x}\mathcal{P}^{k})$ and $\pi^{y}$ of the
  Ces\`{a}ro average of $(\delta_{y}\mathcal{P}^{k})$.
\end{lemma}
\begin{proof}
  By the characterization of averaged mappings 
  \eqref{eq:paafne}, one has
  \begin{align*}
    \mathbb{E}\left[\norm{X-Y}^{2}\right] &\ge \mathbb{E}\left[ 
\norm{T_{\xi_{0}} X -
      T_{\xi_{0}} Y}^{2}\right] + \tfrac{1-\alpha}{\alpha}
    \mathbb{E}\left[\norm{(X-T_{\xi_{0}}X) - (Y-T_{\xi_{0}}Y)}^{2}\right] \\
    & \ge \mathbb{E}\left[ \norm{T_{\xi_{1}}T_{\xi_{0}} X -
        T_{\xi_{1}}T_{\xi_{0}} Y}^{2}\right] \\ &  +
    \tfrac{1-\alpha}{\alpha} \paren{\mathbb{E}\left[
    \norm{(T_{\xi_{0}} X-T_{\xi_{1}}T_{\xi_{0}}X) - 
      (T_{\xi_{0}} Y-T_{\xi_{1}}T_{\xi_{0}}Y)}^{2} \right] + 
    \mathbb{E}\left[\norm{(X-T_{\xi_{0}}X) - (Y-T_{\xi_{0}}Y)}^{2}\right]}\\
    &\ge \dots \\
    & \ge \mathbb{E}\left[ \norm{T_{\xi_{k-1}}\cdots T_{\xi_{0}} X -
        T_{\xi_{k-1}}\cdots T_{\xi_{0}} Y}^{2}\right] \\ &  +
    \tfrac{1-\alpha}{\alpha} \sum_{j=0}^{k-1}\mathbb{E}\left[
    \norm{(T_{\xi_{j-1}}\cdots T_{\xi_{-1}} X-T_{\xi_{j}}\cdots
      T_{\xi_{0}}X) - (T_{\xi_{j-1}}\cdots
      T_{\xi_{-1}}Y-T_{\xi_{j}}\cdots T_{\xi_{0}}Y)}^{2} \right],
  \end{align*}
  where we used $T_{\xi_{-1}}:= \id$ for a simpler representation of
  the sum. The assumption $\mathbb{E}\left[\norm{X_{k}^{X} - X_{k}^{Y}}^{2}\right] =
  \mathbb{E}\left[\norm{X-Y}^{2}\right]$ for all $k \in \mathbb{N}$ then 
implies,
  that for $j=1,\dots,k$ $\mathbb{P}$-a.s.\
  \begin{align*}
    X_{k}^{X}-X_{k-1}^{X} = X_{k}^{Y}-X_{k-1}^{Y} 
    \quad (k \in\mathbb{N}),
  \end{align*}
  and hence by induction
  \begin{align*}
    X_{k}^{X} - X_{k}^{Y} = X-Y.
  \end{align*}
  By disintegrating
  and using $(X,Y)
  \indep (\xi_{k})$ we have $\mathbb{P}$-a.s.\
  \begin{align*}
    0 &= \cex{\norm{(X-X_{k}^{X}) - (Y-X_{k}^{Y} )}^{2}}{X,Y} \\ &=
    \int_{I^{k+1}} \norm{(X-T_{i_{k}} \cdots T_{i_{0}}X) - (Y-T_{i_{k}}
      \cdots T_{i_{0}}Y)}^{2} \mathbb{P}^{\xi}(\dd{i_{k}}) \cdots
    \mathbb{P}^{\xi}(\dd{i_{0}}).
  \end{align*}
 Consequently, for $\mathbb{P}^{(X,Y)}$-a.e.\ $(x,y) \in \Ecal
  \times \Ecal$, we have 
  \begin{align*}
    X_{k}^{x} - X_{k}^{y} = x-y \quad \forall k \in \mathbb{N}
    \quad \mathbb{P}-\mbox{a.s.}
  \end{align*}
  So in particular for any $A \in \mathcal{B}(\Ecal)$
  \begin{align*}
    p^{k}(x,A) = \mathbb{P}(X_{k}^{x} \in A) = \mathbb{P}(X_{k}^{y}
    \in A-(x-y)) = p^{k}(y,A-(x-y))
  \end{align*}
  and hence, denoting $f_{h}=f(\cdot + h)$ and $\nu_{k}^{x} =
  \tfrac{1}{k} \sum_{j=1}^{k} p^{j}(x,\cdot)$, one also has for $f \in
  C_{b}(\Ecal)$ by Theorem \ref{cor:cesaroConvergenceRn}
  \begin{equation*}
    \nu_{k}^{y}f_{x-y} \to \pi^{y} f_{x-y} = \pi_{x-y}^{y} f \quad{and}\quad
    \nu_{k}^{x} f \to \pi^{x}f\mbox{ as }k\to\infty,
  \end{equation*}
  where $ \pi_{x-y} ^{y}:= \pi^{y} (\cdot -
  (x-y))$. So from $\nu_{k}^{y}f_{x-y} = \nu_{k}^{x} f$ for any $f \in
  C_{b}(\Ecal)$ and $k\in\mathbb{N}$ it follows 
  that $\pi_{x-y}^{y} =\pi^{x}$.
\end{proof}

We can now give the proof of the second main result.  
For a given $\mymap{h}{\Ecal \times \Ecal}{\mathbb{R}}$
  we will define sequences of functions $(\hhbar_{k})$ on $\Ecal \times
  \Ecal$ via
  \begin{align*}
    \hhbar_{k}(x,y) \equiv \mathbb{E}\left[h(X_{k}^{x}, X_{k}^{y}) \right], 
    \quad X_{k}^{z}:= T_{\xi_{k-1}} \cdots T_{\xi_{0}}z \mbox{ for any }z \in
  \Ecal\quad (k \in \mathbb{N}).
  \end{align*}
    Note that, by continuity of
  $T_{i}$, $i \in I$ and Lebesgue's dominated convergence theorem,  
  $\hhbar_{k} \in C_{b}(\Ecal\times \Ecal)$ for all
  $k \in \mathbb{N}$ whenever $h\in C_{b}(\Ecal\times \Ecal)$.

\noindent 
\textbf{Proof of Theorem \ref{thm:a-firm convergence Rn}.}
  Let $x,y \in \Ecal$, define $F(x,y)\equiv\norm{x-y}^2$ and the 
  corresponding sequence of functions 
  \begin{align*}
    \Fbar_{k}(x,y) \equiv \mathbb{E}\left[F(X_{k}^{x}, X_{k}^{y}) \right], 
    \quad X_{k}^{z}:= T_{\xi_{k-1}} \cdots T_{\xi_{0}}z \mbox{ for any }z \in
  \Ecal\quad (k \in \mathbb{N}).
  \end{align*}
    By the remarks preceding this proof, 
  $\Fbar_{k} \in C_{b}(\Ecal\times \mathbb{R}^n)$ for all $k \in \mathbb{N}$. 
  From the regularity of $T_{i}$, $i \in I$ and the 
  characterization \eqref{eq:paafne}, we get that a.s.\ for all
  $k \in \mathbb{N}$
  \begin{align}
    \label{eq:averadness_suppPI}
    \norm{X_{k}^{x} - X_{k}^{y}}^{2} \ge \norm{X_{k+1}^{x} -
      X_{k+1}^{y}}^{2} + \tfrac{1-\alpha}{\alpha} \norm{(X_{k}^{x} -
      X_{k+1}^{x}) - (X_{k}^{y} - X_{k+1}^{y})}^{2}.
  \end{align}
  After computing the expectation, this is the same as
  \begin{align*}
    \Fbar_{k}(x,y) \ge \Fbar_{k+1}(x,y) + \tfrac{1-\alpha}{\alpha}
    \mathbb{E}\left[\norm{(X_{k}^{x} - X_{k+1}^{x}) - (X_{k}^{y} -
        X_{k+1}^{y})}^{2}\right].
  \end{align*}
  We conclude that $(\Fbar_{k}(x,y))$ is a monotonically nonincreasing
  sequence for any $x,y\in \Ecal$.\\
  Recall the notation $S_{\pi}\equiv \Supp\pi$ for some measure $\pi$.  
  Let $s, \tilde s \in S_{\pi}$ for the ergodic invariant measure $\pi \in \mathscr{E}$ and define
  the sequence of measures $\gamma_k$ by
  \begin{align*}
    \gamma_{k} f := \mathbb{E}\left[ f(X_{k}^{s}, X_{k}^{\tilde s})\right]
  \end{align*}
  for any measurable function $\mymap{f}{\Ecal \times
    \Ecal}{\mathbb{R}}$. 
   Note that due to nonexpansiveness  the pair $(X_{k}^{s}, X_{k}^{\tilde s})$ a.s. takes
   values in $G_r:=\{(x,y) \in \Ecal \times
    \Ecal: ||x-y||^2 \leq r\}$ for $r =  ||s-\tilde{s}||^2$, so that  
  $\gamma_k$ is concentrated on this set. Since $(X_{k}^{s})$ is a tight
  sequence by Lemma \ref{lemma:tightnessOfNu_n}, and likewise for $(X_{k}^{\tilde{s}})$, 
  we know from  Lemma \ref{lemma:weakCVG_productSpace} that the sequence 
  $(\gamma_k)$ is tight as well.
  Let $\gamma$ be a cluster point of $(\gamma_{k}),$ which is again concentrated 
on $G_{ ||s-\tilde{s}||^2},$   and consider a subsequence
  $(\gamma_{k_j})$ such that $\gamma_{k_j} \rightarrow \gamma.$ By 
   Lemma \ref{lemma:weakCVG_productSpace} we also know that $\gamma \in C(\nu_{1}, 
\nu_{2})$
   where $\nu_1$ and $\nu_2$ are the distributions of the limit 
in convergence in distribution of $(X_{k_j}^{s})$ 
   and $(X_{k_j}^{\tilde{s}}).$
 For any $f \in C_b(\Ecal\times \mathbb{R}^n)$ we have 
 $\gamma_{k_j} f \rightarrow \gamma f.$  So consider the case $f=F^M$ where 
 $F^{M} := \min(M,F)$ for  $M \in \mathbb{R}$.  
 Since $\norm{x-y}^2=F(x,y)=F^M(x,y)$ almost surely (with respect to $\gamma_{k_j}$ and $\gamma$) 
 for $M \ge \norm{s-\tilde s}^{2}$, we have  
 \begin{align*}
    \gamma_{k_j}F = \gamma_{k_j} F^{M} \rightarrow \gamma F^{M}= \gamma F.
  \end{align*}
However,  by the monotonicity in \eqref{eq:averadness_suppPI} we now also 
obtain convergence for the entire sequence:
 \begin{align*}
    \gamma_{k}F = \gamma_{k} F^{M}  \searrow \gamma F^{M}=\gamma F.
  \end{align*}
  Let $(X,Y) \sim \gamma$ and $(\tilde \xi_{k}) \indep (\xi_{k})$ be
  another i.i.d.\ sequence with $(X,Y) \indep (\tilde \xi_{k}),
  (\xi_{k})$. We use the notation 
  $\tilde X_{k}^{x} := T_{\tilde \xi_{k-1}} \cdots T_{\tilde\xi_{0}} x$, $x \in \Ecal.$ 
  Define the sequence of functions 
  \[
   \Fbar^M_k(x,y) \equiv \mathbb{E}\left[F^M(X_{k}^{x}, X_{k}^{y}) \right] \quad (k \in \mathbb{N}),
  \]
  and note that  $\Fbar^{M}_{k}  \in C_b(\Ecal\times \Ecal)$.  
  When $M \ge \norm{s-\tilde s}^{2}$ this yields 
  \begin{align*}
 \gamma \Fbar_k = \gamma \Fbar^{M}_{k} &= 
 \mathbb{E}\left[ \min \left(M,\norm{\tilde X_{k}^{X} - \tilde X_{k}^{Y}}^{2}\right)\right] 
 = \lim_{j \to \infty} \gamma_{ k_{j}} \Fbar^{M}_{k} \\ 
 &=  \lim_{j \to \infty} \mathbb{E}\left[ 
      \min\left(M,\norm{\tilde
            X_{k}^{X_{ k_{j}}^{s}} - 
            \tilde X_{k}^{X_{k_{j}}^{\tilde s}}}^{2}%
            \right) 
                \right] \\
    &= \lim_{j \to \infty} \mathbb{E}\left[ \min\left(M,\norm{
          X_{k+ k_{j}}^{s} - X_{k+ k_{j}}^{\tilde s}}^{2}
      \right) \right] \\ 
      &= \lim_{j \to \infty} \gamma_{k+ k_{j}} F^{M} = 
      \gamma F^{M} = \gamma F.
  \end{align*}
This means that  for all $k \in \mathbb{N},$
  \begin{align*}
    \mathbb{E}\left[\norm{X_{k}^{X} - X_{k}^{Y}}^{2}\right] =
    \mathbb{E}[\norm{X-Y}^{2}]. 
  \end{align*}
  For $\mathbb{P}^{(X,Y)}$-a.e.\ $(x,y)$ we have $x,y \in S_{\pi}$ and thus 
  $\pi^x=\pi^y=\pi$ where $\pi^x$ is the unique ergodic measure with $x\in 
S_{\pi^{x}}$
 (see Remark \ref{cor:weakCvgMetricSpace}).
 An application of Lemma \ref{lemma:const_dist} then yields  
 $\pi(\cdot) = \pi(\cdot-(x-y))$, i.e.\ $x=y$. Hence $X=Y$ a.s. implying
  $\nu_{1}=\nu_{2}=:\nu$ and $\gamma F = 0$. That means
  \begin{align*}
    \gamma_{k}F = \mathbb{E}\left[\norm{X_{k}^{s}
        - X_{k}^{\tilde{s}}}^{2}\right] \to 0 \qquad \text{as } k \to \infty.
  \end{align*}
  Now Lemma \ref{lemma:prokhorovDist_properties} yields
  \begin{align*}
    \mathbb{P}\left(\norm{X_{k}^{s} - X_{k}^{\tilde{s}}} >
    \epsilon\right) \le \frac{\mathbb{E}\left[\norm{X_{k}^{s} -
    X_{k}^{\tilde{s}}}\right]}{\epsilon} \le \frac{\mathbb{E}\left[\sqrt{\norm{X_{k}^{s} -
    X_{k}^{\tilde{s}}}^{2}}\right]}{\epsilon} \to 0
  \end{align*}
  as $k \to \infty$ for any $\epsilon > 0$;  so this yields convergence 
  of the corresponding probability  measures $\delta_{s} \mathcal{P}^{k}$ and 
  $\delta_{\tilde{s}} \mathcal{P}^{k}$  in the Prokhorov-L\`evy metric:
  \begin{align*}
    d_P(\delta_{s} \mathcal{P}^{k},\delta_{\tilde{s}} \mathcal{P}^{k}) 
    \rightarrow 0.
  \end{align*}
By the triangle inequality, therefore, if 
  $\delta_{s} \mathcal{P}^{k_j} \rightarrow \nu$,  then also
  $\delta_{\tilde{s}} \mathcal{P}^{k_j} \rightarrow \nu$ for any
  $\tilde{s} \in S_{\pi}$.  Hence 
  \begin{align*}
    d_{P}(\delta_{\tilde{s}}\mathcal{P}^{k_{j}},\nu) &\le
    d_{P}(\delta_{s}\mathcal{P}^{k_{j}},
    \delta_{\tilde{s}}\mathcal{P}^{k_{j}}) +
    d_{P}(\delta_{s}\mathcal{P}^{k_{j}}, \nu) \to 0,\quad \text{as
    } j \to \infty.
  \end{align*}
  By Lebesgue's dominated convergence theorem we conclude 
  that, for any $f \in C_{b}(\Ecal)$ and $\mu \in
  \mathscr{P}(S_{\pi})$,
  \begin{align*}
    \mu \mathcal{P}^{k_{j}} f = \int_{S_{\pi}}
    \delta_s \mathcal{P}^{k_{j}}f \mu(\dd{s}) \to \nu f, \quad
    \text{as } j \to \infty.
  \end{align*}
In particular, 
  $\mu \mathcal{P}^{k_{j}} \to \nu$ and taking $\mu=\pi$ yields  
  $\nu = \pi$.  Thus, all cluster points of
  $(\delta_{s}\mathcal{P}^{k})$ for all $s \in S_{\pi}$ have the same
  distribution $\pi$ and hence, because the sequence is tight, 
  $\delta_{s} \mathcal{P}^{k}=p^{k}(x,\cdot) \to \pi$.

  Now, let $\mu \in \mathscr{P}(S)$, where $S = \bigcup_{\pi \in
    \mathscr{E}} S_{\pi}$. By what  we have just shown we have for $x\in
  \Supp \mu$, that $p^{k}(x,\cdot) \to \pi^{x}$, where $\pi^{x}$ is
  unique ergodic measure with $x\in S_{\pi^{x}}$. Then, again by Lebesgue's
  dominated convergence theorem, one has for any $f \in C_{b}(\Ecal),$
  \begin{align}
  \label{eq:muconv}
    \mu \mathcal{P}^{k} f = \int f(y) p^{k}(x, \dd{y}) \mu(\dd{x}) \to
    \int f(y) \pi^{x}(\dd{y}) \mu(\dd{x}) =:  \pi^{\mu} f   
    \mbox{ as }k \to \infty,  
  \end{align}
and the measure $\pi^{\mu} $ is again invariant
  for $\mathcal{P}$ by invariance of $\pi^{x}$ for all $x \in S.$
  Now, let $\mu = \delta_{x}$, $x \in \Ecal \setminus S$. We
  obtain the tightness of $(\delta_{x}\mathcal{P}^{k})$ from the
  tightness of $(\delta_{s}\mathcal{P}^{k})$ for $s \in S$.  Indeed,  for
  $\epsilon >0$ there exists a compact
  $K_{\epsilon} \subset \Ecal$ with
  $p^{k}(s,K_{\epsilon})> 1-\epsilon$ for all $k \in \mathbb{N}.$
  This together with the fact that $T_{i}$, $i \in I$ is nonexpansive implies that 
  $\norm{X_{k}^{x}-X_{k}^{s}}\le \norm{x-s}$ for all
  $k \in \mathbb{N}$ hence
  $p^{k}(x,\cb(K_{\epsilon},\norm{x-s}))>1-\epsilon$, where $p$ is the 
  transition kernel defined by \eqref{eq:trans kernel}. 
  Tightness  implies the existence of a cluster point $\nu$ of the sequence
  $(\delta_{x}\mathcal{P}^{k})$. From Theorem \ref{cor:cesaroConvergenceRn}
  we know that $\nu_{k}^{x} = \tfrac{1}{k} \sum_{j=1}^{k}
  \delta_{x}\mathcal{P}^{j} \to \pi^{x}$ for some $\pi^{x} \in \inv
  \mathcal{P}$ with $S_{\pi^x} \subset S.$  Furthermore, we have $\nu \in
  \mathscr{P}(S_{\pi^{x}}) \subset \mathscr{P}(S) $ by
  Lemma \ref{thm:invMeas_nonexpans_map}\eqref{item:nonExpProp_4}. So by 
  \eqref{eq:muconv} there exists $\pi^{\nu} \in \inv \mathcal{P}$ with
  $\nu\mathcal{P}^{k} \to \pi^{\nu}$.
  
  In order to complete the proof we have to show that $\nu = \pi^{x}$,
  i.e.\ $\pi^{x}$ is the unique cluster point of
  $(\delta_{x}\mathcal{P}^{k})$ and hence convergence follows by
  Proposition \ref{thm:cvg_subsequences}.  It suffices to  show that
  $\pi^{\nu} = \pi^{x}$, since then, as $k \to \infty$
  \begin{align*}
    d_{P}(\nu, \pi^{x}) = \lim_{k} d_{P}(\delta_{x}\mathcal{P}^{k},
    \pi^{x}) = \lim_{k} d_{P}(\delta_{x}\mathcal{P}^{k+j}, \pi^{x}) =
    d_{P}(\nu \mathcal{P}^{j}, \pi^{x}) = d_{P}(\nu \mathcal{P}^{j}, \pi^{\nu}) 
\to 0.
  \end{align*}
To begin, fix  $\pi \in
  \inv\mathcal{P}$.  For any $\epsilon>0$ let $A_{k} := \{X_{k}^{x} \in 
\cb(S_{\pi},
  \epsilon)\}.$  By nonexpansivity $A_{k} \subset A_{k+1}$ for $k \in
  \mathbb{N}$, since we have by Lemma \ref{lemma:support_invariant_distr}
  a.s.
  \begin{align*}
    \dist(X_{k+1}^{x}, S_{\pi}) \le \dist(X_{k+1}^{x}, T_{\xi_{k}}
    S_{\pi}) \le \dist(X_{k}^{x},S_{\pi}). 
  \end{align*}
  Hence $(p^{k}(x, \cb(S_{\pi}, \epsilon))) = (\mathbb{P}(A_{k}))$ is
  a monotonically increasing sequence and bounded from above and
  therefore the sequence converges to some $b_{\epsilon}^{x} \in
  [0,1]$ as $k\to \infty$. It
  follows 
  \begin{align}
  \label{eq:b_expconv}
  b_{\epsilon}^{x} =  \lim_{k} p^{k}(x,\cb(S_{\pi},\epsilon)) = \lim_{k}
    \frac{1}{k}\sum_{j=1}^{k} p^{j}(x,\cb(S_{\pi},\epsilon)).
  \end{align}
  and thus
  $\nu(\cb(S_{\pi},\epsilon)) = \pi^{x}(\cb(S_{\pi},\epsilon))$ 
  for all $\epsilon$, which make $\cb(S_{\pi},\epsilon)$ both
  $\nu$- and
  $\pi^{x}$-continuous. Note that there are at most countably many
  $\epsilon>0$ for which this may fail, see \cite[Chapter 3, Example
  1.3]{kuipers1974uniform}). 
 
  With the same argument used for  \eqref{eq:b_expconv} we also obtain for any 
  $k \in \mathbb{N}$ that $\nu \mathcal{P}^{k}(\cb(S_{\pi},\epsilon)) =  
\pi^{x}(\cb(S_{\pi},\epsilon))$
  with only countably many $\epsilon$ excluded, 
 and so 
  \begin{align*}
    \pi^{\nu}(\cb(S_{\pi},\epsilon)) = \pi^{x}(\cb(S_{\pi},\epsilon))
  \end{align*}
  also needs to hold for all except countably many $\epsilon$.  
  Since
  $\pi^{\nu} \in \inv \mathcal{P}$,  this implies 
  that $\pi^{\nu} = \pi^{x}$ by
  Lemma \ref{lemma:equalityBallsimpliesEqOfMeasures} combined with
  Remark \ref{rem:equality}.
  For a general initial measure $\mu_0 \in \mathscr{P}(\Ecal)$,
  one has, yet again by Lebesgue's dominated convergence theorem, that
  \begin{align*}
    \mu_0 \mathcal{P}^{k} f = \int f(y) p^{k}(x, \dd{y}) \mu_0(\dd{x}) \to
    \int f(y) \pi^{x}(\dd{y}) \mu_0(\dd{x}) =: \pi^{\mu_0} f,
  \end{align*}
  where $\pi^{x}$ denotes the limit of $(\delta_{x} \mathcal{P}^{k})$
  and the measure $ \pi^{\mu_0}$ is again invariant for $\mathcal{P}$.
  This completes the proof.
\hfill $\Box$

\begin{rem}[a.s.\ convergence]\label{rem:asCvgSequence}
  If we were to choose $X$ and $Y$ in \eqref{eq:averadness_suppPI} such that
  $\mathcal{L}(X),\mathcal{L}(Y) \in \mathscr{P}(S_{\pi})$, where $\pi
  \in \mathscr{E}$, then still $\gamma_{k} F \to \gamma F = 0$,
  where $\gamma \in C(\pi,\pi)$. For $(W,Z) \sim \gamma$ it
  still holds that $W=Z$ and hence
  \begin{align*}
    \norm{X_{k}^{X} - X_{k}^{Y}} \to 0 \qquad \text{a.s.}
  \end{align*}
  by monotonicity of $(\gamma_{k}F)$.
\end{rem}

\section{Examples: Stochastic Optimization and Inconsistent Convex Feasibility}
\label{sec:incFeas}

To fix our attention we focus on the following optimization problem
  \begin{equation}\label{eq:opt prob}
  \underset{\mu\in \mathscr{P}_2(\Ecal)}{\mbox{minimize}}
  \int_{\Ecal}\mathbb{E}_{\xi}[f_{\xi^f}(x) + g_{\xi^g}(x)]\mu(dx). 
  \end{equation}
The random variable with values on $I_f\times I_g$ are denoted $\xi = (\xi^f, \xi^g)$. 
This model covers deterministic composite optimization as a special case: $I_f$ and $I_g$ consist
of single elements and the measure $\mu$ is a point mass.  

The algorithms reviewed in this section rely on proximal, or simply {\em prox}, mappings 
of the functions
$f_i$ and $g_i$, denoted $\prox_{f_i}$ and $\prox_{g_i}$.   
For proper, lsc and convex functions $\mymap{f}{\Ecal}{\mathbb{R}\cup\{+\infty\}}$, 
the proximal mapping \cite{Moreau65} 
defined by
\begin{equation}\label{eq:prox}
 \prox_{f}(x)\equiv\argmin_{y}\{f(y)+ \tfrac{1}{2}\norm{y-x}^2\}.
\end{equation}

\subsection{Stochastic convex forward-backward splitting}\label{ex:spg ncvx}
We begin with a general prescription of the forward-backward splitting algorithm together with 
abstract properties of the corresponding fixed point mapping, and then specialize this to 
more concrete instances.   It is assumed throughout this subsection 
that $\mymap{f_{i}}{\Ecal}{\mathbb{R}}$ is continuously 
  differentiable and convex for all $i \in I_f$ and that the extended-valued 
  function
  $\mymap{g_{i}}{\Ecal}{\mathbb{R}\cup\{+\infty\}}$
  is proper, lower semi-continuous (lsc) and convex for all $i \in I_g$.   

\begin{algorithm}    
\SetKwInOut{Output}{Initialization}
  \Output{Set $X_{0} \sim \mu_0 \in \mathscr{P}_2(\Ecal)$, 
  $X_0\sim \mu$,  ${t}>0$, 
   and $(\xi_{k})_{k\in\Nbb}$ 
  another i.i.d.\ sequence with values on $I_f\times I_g$ and  
  $X_0 \indep (\xi_{k})$.}
    \For{$k=0,1,2,\ldots$}{
            { 
            \begin{equation}\label{eq:spcd}
                X_{k+1}= T^{FB}_{\xi_k}X_k\equiv 
                \prox_{tg_{\xi^g_k}}\paren{X_{k}-{t}\nabla f_{\xi^f_k}(X_{k})}
            \end{equation}
            }\\
    }
  \caption{Stochastic Convex Forward-Backward Splitting}\label{algo:sfb}
\end{algorithm}

When $f_{\xi^f}(x) = f(x) + \eta_{\xi^f}\cdot x$ and $g_{\xi^g}$ is the zero function, 
then this is just 
 steepest descents with linear noise discussed in Section \ref{sec:consist RFI}.  More generally, 
 \eqref{eq:spcd} with $g_{\xi^g}$ the zero function models stochastic gradient 
 descents, which is a central algorithmic template in many applications.  To date, 
 convergence results for these types of methods are limited to ergodic results or to 
 a.s. convergence.  As Proposition \ref{thm:Hermer} shows, almost sure convergence requires
 the existence of a common fixed point of all the randomly selected operators.  Ergodic results 
 on the other hand only provide access to one of the moments of the limiting distribution.  Our analysis 
 provides for information on all moments in the limit and does not require common fixed points. 

 For the next result it is helpful to recognize the forward-backward mapping as the 
 composition of two mappings:  $T^{FB}_i\equiv \prox_{t g_i}\circ T^{GD_t}_i$ where 
 $T^{GD_t}_i\equiv \Id - t\nabla f_i$. 
\begin{prop}\label{t:sfb}
In addition to the standing assumptions, suppose that for all 
$i\in I_f$, $\nabla f_i$ is Lipschitz continuous with constant 
 $L$ on  $\Ecal$.
 Then for all $\alpha\in (0,1)$ and all 
 step lengths $t\in \left(0,\tfrac{2\alpha}{L}\right]$ the following hold.
 \begin{enumerate}[(i)]
 \item\label{ex:spg cvx i} $T^{GD_t}_i$ is averaged on $\Ecal$ with constant $\alpha$;  
  \item\label{ex:spg cvx ii} $T^{FB}_i$ is averaged on $\Ecal$ with constant 
  $\alphabar = \frac{2}{1+2/\max\{tL, 1\}}$ for all $i\in I$.
  \item\label{ex:spg cvx iii} 
  Whenever there exists an invariant measure 
  for the Markov operator $\mathcal{P}$ corresponding to \eqref{eq:spcd},  
  the distributions of the sequences of random variables
  converge to an invariant measure in the Prokhorov-L\`evy metric.  
 \end{enumerate} 
\end{prop}

\begin{proof}
\eqref{ex:spg cvx i}.  
By \cite[Corollaire 10]{BaiHad77}
		\begin{eqnarray*}
            \tfrac{1}{L}\|\nabla f(x) - \nabla f(y)\|^2&\leq& 
                \ip{\nabla f(x) - \nabla f(y)}{x-y}
		\end{eqnarray*}
		For $t = \tfrac{2\alpha}{L}$  we have 
		$2t = \tfrac{t^2L}{\alpha}$, and for all $x,y\in \Ecal$ 
		\begin{eqnarray*}
            &\tfrac{t^2L}{\alpha}\tfrac{1}{L}\|\nabla f(x) - \nabla f(y)\|^2\leq 
                2t\ip{\nabla f(x) - \nabla f(y)}{x-y}&\\
                &\iff&\\
            &\tfrac{1}{\alpha}\|t\nabla f(x) - t\nabla f(y)\|^2\leq 
                2\ip{t\nabla f(x) - t\nabla f(y)}{x-y}&\\
               &\iff&\\
            &\norm{x-y}^2 + 
            \paren{1+\tfrac{1-\alpha}{\alpha}}\|t\nabla f(x) - t\nabla f(y)\|^2&\\
            &\qquad \leq 
                 2\ip{t\nabla f(x) - t\nabla f(y)}{x-y} + \norm{x-y}^2&\\
               &\iff&\\
            &\norm{\paren{x - t\nabla f(x)} - \paren{y - t\nabla f(y)}}^2 
             &\\
            &\qquad \leq \norm{x-y}^2 - \tfrac{1-\alpha}{\alpha}\|t\nabla f(x) - t\nabla f(y)\|^2&\\
               &\iff&\\
            &\norm{T_{GD_t}x - T_{GD_t}y}^2 
            \leq \norm{x-y}^2 - \tfrac{1-\alpha}{\alpha}\psi(x, y, T_{GD_t}x, T_{GD_t}y),&           
        \end{eqnarray*}
where the last implication follows from \eqref{eq:nice ineq}.

\eqref{ex:spg cvx ii}  Since $g_i$ is proper, convex and lsc for all $i$, the prox mapping 
is well-defined and averaged with constant $\alpha=1/2$ on $\Ecal$.  The rest follows from 
part \eqref{ex:spg cvx i} and the calculus of 
compositions of averaged mappings 
\cite[Proposition 4.32]{BauCom11}.

\eqref{ex:spg cvx iii}.  This follows from part 
\eqref{ex:spg cvx ii} and Theorem \ref{thm:a-firm convergence Rn}. 
\end{proof}

Note that an upper bound on the step size $t$ is $2/L$; the price to pay for taking 
such a large step is the loss of the averaging property (averaging constant $\alpha=1$). 
The result also captures stochastic gradient descent as a special case:
$g_i\equiv 0$ for all $i$.  
The assumptions of Proposition \ref{t:sfb} are not unusual, but weaker than
the standard assumption of strong convexity \cite{DieuDurBac20}.  
The generality of global convergence and the information 
that this yields is new:  convergence is to a distribution, not just the expectation.  
The result narrows the work of proving convergence of stochastic forward-backward 
algorithms to verifying that $\inv \Pcal$ is nonempty.
The next corollary shows how this
is done for the special case of stochastic gradient descent. 
\begin{prop}[existence of invariant measures and convergence: stochastic gradient descent] 
  In problem \eqref{eq:opt prob} let  $g_i(x)\equiv 0$ for all $i$ at each $x$.
  In addition to assumptions of Proposition \ref{t:sfb}, suppose that 
  \begin{enumerate}[(i)]
   \item $\nabla f_i$ is strongly monotone on $\Ecal$ with the same 
   constant for all $i\in I$:
  \begin{equation}
 \label{e:strmonotone} 
    \exists \tau_f>0:\quad \forall i\in I, \quad\tau_f\|x-y\|^2\leq \ip{\nabla f_i(x)-\nabla f_i(y)}{x-y}
 \quad \forall x, y\in \Ecal;
\end{equation}
   \item the expectation $\mathbb{E}\ecklam{f_{\xi }(x)}$ 
  attains a minimum at $\xbar\in \Ecal$ with value 
  $\mathbb{E}\ecklam{f_{\xi }(\xbar)}= \bar p$;
  \item $\mathbb{E}\left[\norm{X_{0} -\bar x}^{2}\right]<\infty$ where $X_0$ is 
  a random variable with distribution $\mu_0\in \mathscr{P}_2(\Ecal)$.
  \end{enumerate}
  Then for any $t\in (0,\min\{1/L, 1/\tau_f,\tfrac{\tau_f}{L^2}\}]$, 
  the Markov operator corresponding to stochastic gradient descent with 
  fixed step length $t$  possesses invariant measures, and, when initialized with $\mu_0$,
  the distributions of the iterates converge in the Prokhorov-L\`evy 
  metric to an invariant measure.
  \end{prop}
\begin{proof}
  In this case $T^{FB}_i = T^{GD_t}_i\equiv \Id - t\nabla f_i$. 
  By \cite[Proposition 3.6]{LukNguTam18}, for any step size 
  $t\leq \tfrac{\tau_f}{L^2}$, $T^{GD_t}$ is averaged
  with constant $\alpha=1/2$ on $\Ecal$.  Convergence 
  to an invariant measure, whenever this exists, then follows from 
  Theorem \ref{thm:a-firm convergence Rn}.

  It remains to show that the 
  corresponding Markov operator possesses invariant distributions.
  To establish this, 
  note that  
  \begin{align}
    \norm{X_{k+1} - \bar x}^{2} =& 
    \norm{X_{k} - {t} \nabla f_{\xi _{k}}(X_{k}) - \bar x}^{2} - 
      \norm{X_{k+1} - 
    X_{k} - {t} \nabla f_{\xi_{k}}(X_{k})}^{2} \nonumber\\
    &= \norm{X_{k}-\bar x}^{2} - \norm{X_{k+1}-X_{k}}^2 
    -2{t}\act{\nabla f_{\xi_{k}}(X_{k}), X_{k+1} -X_{k}+X_{k}-\bar x}.
    \label{e:Amster}
  \end{align}
  For functions with Lipschitz continuous gradients the following {\em growth condition}
  holds
  \begin{align}\label{e:dam}
    \act{\nabla f_{\xi_{k}}(X_{k}),X_{k+1}-X_{k}} \ge
    f_{\xi_{k}}(X_{k+1})-f_{\xi_{k}}(X_{k}) 
-\frac{L}{2}\norm{X_{k+1}-X_{k}}^{2}.
  \end{align}
  The assumption of strong monotonicity of the gradients implies \cite[Proposition 2.2]{LukShe18}
 \[
  \ip{\nabla f_{\xi_k}(X_k)}{X_k-\xbar}\geq f_{\xi_k}(X_k) - f_{\xi_k}(\xbar) + 
  \tfrac{\tau_f}{2}\norm{X_k-\xbar}^2.
 \]
 Taking the expectation and interchanging the gradient with the expectation yields
  \begin{align}\label{e:aged}
    \act{\nabla \mathbb{E}\ecklam{f_{\xi}(X_{k})}, X_{k} - \bar x} \ge 
    \mathbb{E}\ecklam{f_{\xi}(X_{k})}-\bar p + 
    \frac{\tau_f}{2}\norm{X_{k}-\bar x}^{2}.
  \end{align}
Putting \eqref{e:Amster}-\eqref{e:aged} together yields 
  \begin{align*}
    \mathbb{E}\left[\norm{X_{k+1} - \bar x}^{2}\right] &\le
    (1-{t}\tau_f)\mathbb{E}\left[\norm{X_{k} - \bar x}^{2} \right] -
    (1-{t} L) \mathbb{E}\left[\norm{X_{k+1} - X_{k}}^{2}\right] -
    2{t} \paren{\mathbb{E}[f_{\xi_{k}}(X_{k+1})] - \bar p} \\ 
    & \le     (1-{t}\tau_f)\mathbb{E}\left[\norm{X_{k} - \bar x}^{2}
    \right] +2 {t} \bar p\quad\forall t\in(0,\min\{1/L, 1/\tau_f\}].
  \end{align*}
  Thus, whenever $t\in(0,\min\{1/L, 1/\tau_f\}]$
  \begin{align*}
    \mathbb{E}\left[\norm{X_{k} - \bar x}^{2} \right]
    \le (1-{t}\tau_f)^{k} \mathbb{E}\left[\norm{X_{0} - \bar x}^{2}
    \right] + 2t \bar p  \sum_{i=0}^{k-1} (1-{t} \tau_f)^{i} \le
    \mathbb{E}\left[\norm{X_{0} - \bar x}^{2} \right] + 
    \frac{2\bar p }{\tau_f}.
  \end{align*}
  So the sequence $\left( \mathbb{E}\left[\norm{X_{k} - \bar x}^{2} \right]
  \right)_{k\in\Nbb}$ is bounded since, by assumption,  $\mathbb{E}\left[\norm{X_{0} -
      \bar x}^{2}\right]$ is bounded.   The existence of an
  invariant measure then follows from Theorem \ref{thm:ex_inRn}.
\end{proof}
  As noted in \cite[pp1171]{LukNguTam18}, the step $t$ in the stochastic gradient 
  could be taken much larger than the conservative estimates given here.  A justification
  of this is beyond the scope of this study.

\subsection{Stochastic Douglas-Rachford}  
Another prevalent algorithm for nonconvex problems is the 
Douglas-Rachford algorithm \cite{LionsMercier79}.  This is
based on compositions of {\em reflected prox mappings}:
\begin{equation}\label{eq:Rprox}
 R_{f}\equiv 2\prox_f -\Id.
\end{equation}
For this algorithm we assume only convexity of the constituent functions. 
\begin{algorithm}    
\SetKwInOut{Output}{Initialization}
  \Output{Set $X_{0} \sim \mu_0 \in \mathscr{P}_2(\Ecal)$, 
  $X_0\sim \mu$,   
   and $(\xi_{k})_{k\in\Nbb}$ 
  another i.i.d.\ sequence with $\xi_{k} = (\xi^f_{k}, \xi^g_{k})$ 
  taking values on $I_f\times I_g$ and  
  $X_0 \indep (\xi_{k})$.}
    \For{$k=0,1,2,\ldots$}{
            { 
            \begin{equation}\label{eq:sdr}
                X_{k+1}= T^{DR}_{\xi_k}X_k\equiv 
                \frac{1}{2}\paren{R_{f_{\xi^f_k}}\circ R_{g_{\xi^g_k}} + \Id}(X_{k})
            \end{equation}
            }\\
    }
  \caption{Stochastic Douglas-Rachford Splitting}\label{algo:sdr}
\end{algorithm}
Algorithm \ref{algo:sdr} has been studied for solving large-scale, convex optimization 
and monotone inclusions (see for example \cite{Pesquet19, Cevher2018}).
\begin{prop}\label{t:sdr}
Suppose that 
for all $i=(i_1, i_2)\in I_f\times I_g$ the extended-valued functions 
 $\mymap{f_{i_1}}{\Ecal}{\mathbb{R}\cup\{+\infty\}}$ and  
 $\mymap{g_{i_2}}{\Ecal}{\mathbb{R}\cup\{+\infty\}}$ are 
 proper, lsc and convex on $\Ecal$. 
 Then whenever there exists an invariant measure 
  for the Markov operator $\mathcal{P}$ corresponding to \eqref{eq:sdr}, 
  the distributions of the sequences of random variables
  converge to an invariant measure in the Prokhorov-L\`evy metric.  
\end{prop}
\begin{proof}
This follows immediately from 
\cite[Proposition 2]{LionsMercier79} (which establishes that $T^{DR}_i$ is 
averaged with constant $\alpha=1/2$ on $\Ecal$ for all $i$) and 
Theorem \ref{thm:a-firm convergence Rn}.
\end{proof}

\subsection{Inconsistent set feasibility}\label{ex:2linearSpaces_withInvMeas}
We conclude this study with our explanation for the numerical behavior
observed in Fig. \ref{fig:Axb}.  This is an affine feasibility 
problem:
\begin{equation}\label{eq:feas}
 \mbox{Find}\quad x\in L\equiv \cap_{j\in I} \{x~|~\ip{a_j}{x}=b_j\}.
\end{equation}
When the intersection is empty
we say that the problem is {\em inconsistent}.  
Consistent or not, we apply the method of cyclic projections \eqref{eq:CP}. 
Even though the projectors onto the corresponding problems have an analytic 
expression, this representation can only be evaluated to finite precision.  The 
trick here is to view the algorithm not as inexact cyclic projections onto
deterministic hyperplanes, but rather as {\em exact} projections onto 
randomly selected hyperplanes.  

Indeed, consider the following generalized affine noise model for a single
   affine subspace: $H^{(\xi,\zeta)}_\xbar = \mysetc{x \in \Rn}{\act{a+\xi, x-\xbar} = \zeta}$, 
 where $a \in \Rn$  
    and $\xbar$ satisfies $A\xbar=b$ for a given $b \in \mathbb{R}$ and noise 
  $(\xi,\zeta) \in \Rn\times\mathbb{R}$ is independent.  
  The key conceptual distinction is that 
  the analysis proceeds with {\em exact} projections onto randomly selected
  hyperplanes $H^{(\xi,\zeta)}_\xbar$, rather than working with inexact projections 
  onto deterministic hyperplanes.    

  The main result of this section uses the following result about the more familiar contractive 
mappings.  
\begin{prop}
\label{thm:contraInExpec}
   Let
   $\mymap{T_i}{\Ecal}{\Ecal}$ for $i\in I$ and let 
  $\mymap{\Phi}{\Ecal \times I }{\Ecal}$ be given by
  $\Phi(x,i)\equiv T_i(x)$.  
  Denote by  $\mathcal{P}$ the Markov
  operator with  update function $\Phi$ and transition kernel $p$ defined by 
  \eqref{eq:trans kernel}.  
  Suppose that $\Phi$ is a 
  \emph{contraction in expectation} with 
  constant $r<1$, i.e.\  
  $\mathbb{E}[\|\Phi(x,\xi) - \Phi(y,\xi)\|^2] \le r^2 \|x-y\|^2$ for all $x,y \in
  \Ecal$. 
  Suppose in addition  that  
  there exists $y \in \Ecal$ with
  $\mathbb{E}[\|\Phi(y,\xi) - y\|^2] < \infty$.  
Then the following hold. 
\begin{enumerate}[(i)]
 \item\label{thm:contraInExpec i} There exists a unique
  invariant measure $\pi \in \mathscr{P}_{2}(\Ecal)$ for $\mathcal{P}$ and 
  \begin{align*}
    W_2(\mu_0 \mathcal{P}^{n} , \pi) \le r^{n} W_2(\mu_0,\pi)
  \end{align*}
  for all $\mu_0 \in \mathscr{P}_{2}(\Ecal)$; that is, the sequence $(\mu_k)$ defined 
  by $\mu_{k+1}=\mu_k\mathcal{P}$ converges  to $\pi$ linearly (geometrically) from any initial 
  measure $\mu_0\in \mathscr{P}_{2}(\Ecal)$.
\item\label{thm:contraInExpec ii}   
$\Phi$ is averaged
in expectation with constant $\alpha = (1+r)/2$:
  \begin{eqnarray}\label{eq:paafne i.e.}
&&\mbox{ for }\alpha = (1+r)/2, \quad \forall x, y \in \Ecal,\\
&&\quad\mathbb{E}\ecklam{\norm{\Phi(x,\xi)-\Phi(y,\xi)}^2}\leq 
\norm{x-y}^2 - 
\tfrac{1-\alpha}{\alpha}\mathbb{E}\left[\psi(x,y, \Phi(x,\xi), \Phi(y,\xi))\right].
\nonumber
\end{eqnarray}
\end{enumerate}
\end{prop}
\begin{proof}
  Note that for any pair of distributions $\mu_1,\mu_2 \in
  \mathscr{P}_{2}(\Ecal)$ and an 
  optimal coupling $\gamma\in C_*(\mu_1,\mu_2)$ (possible by
  Lemma \ref{lemma:WassersteinMetric_prop}) it holds that 
  \begin{eqnarray*}
    W_2^2(\mu_1 \mathcal{P}, \mu_2 \mathcal{P}) &\le& 
    \int_{\Ecal\times \Ecal}\mathbb{E}[d^2(\Phi(x,\xi),\Phi(y,\xi))]\ \gamma(d x, d y) 
    \\
    &\le& r ^2\int_{\Ecal\times \Ecal}d^2(x,y)\ \gamma(d x, d y) = r^2 W_2^2(\mu_1,\mu_2),
  \end{eqnarray*}
where $\xi$ is independent of $\gamma$.
Moreover, $\mathcal{P}$ is a self-mapping on  ${\mathscr{P}_{2}(\Ecal)}$.  
To see this   
let  $\mu \in \mathscr{P}_{2}(\Ecal)$  independent of $\xi$
and let $y$ be a point in $\Ecal$ where $\mathbb{E}[\|\Phi(y,\xi) - y\|^2] < \infty$.   
Then by the triangle inequality and the contraction property
  \begin{eqnarray*}
    &&\int_\Ecal\mathbb{E}[\|\Phi(x,\xi)-y\|^2]\ \mu(dx) 
    \nonumber \\
    &&\qquad \le 4\paren{ \int_\Ecal\mathbb{E}[\|\Phi(x,\xi) - \Phi(y,\xi)\|^2]\ \mu(dx) +
   \mathbb{E}[\|\Phi(y,\xi) - y\|^2]}\\
    &&\qquad \le 4\paren{\int_\Ecal r^2 \|x-y\|^2\ \mu(dx) +
   \mathbb{E}[\|\Phi(y,\xi) - y\|^2]}<\infty.
  \end{eqnarray*}
Therefore $\mu \mathcal{P} \in \mathscr{P}_{2}(\Ecal)$.  Altogether, this 
 establishes  that $\mathcal{P}$ is a contraction on the separable complete 
metric space $(\mathscr{P}_{2}(\Ecal), W_2)$ and hence Banach's Fixed Point Theorem
  yields existence and uniqueness of $\inv\mathcal{P}$ and linear convergence 
  of the fixed point sequence.  
  
  To see \eqref{thm:contraInExpec ii}, note that, by \eqref{eq:nice ineq},
  \begin{eqnarray}
  \mathbb{E}[\psi(x,y, T_\xi x, T_\xi y)]&=& 
  \mathbb{E}\left [\|(x - \Phi(x,\xi)) - (y - \Phi(y,\xi))\|^2\right]\nonumber\\
  &=&\|x-y\|^2+
  \mathbb{E}\left[\|\Phi(x,\xi)-\Phi(y,\xi)\|^2 - 2\langle x-y, \Phi(x,\xi)-\Phi(y,\xi)\rangle\right] 
  \nonumber\\
\label{eq:corona home}   &\leq& (1+ r)^2\|x-y\|^2,
  \end{eqnarray}
where the last inequality follows from the Cauchy-Schwarz inequality and 
the fact that $\Phi(\cdot,\xi)$ is a contraction in expectation.  Again using the 
contraction property and \eqref{eq:corona home} we have
\begin{eqnarray*}
 \mathbb{E}\left[\|\Phi(x,\xi)-\Phi(y,\xi)\|^2\right]&\leq& \|x-y\|^2 - (1-r^2)\|x-y\|^2 \\
 &\leq&\|x-y\|^2 - \tfrac{1-r^2}{(1+r)^2}\mathbb{E}[\psi(x,y, T_\xi x, T_\xi y)]. 
\end{eqnarray*}
The right hand side of this inequality is just the characterization 
\eqref{eq:paafne} of mappings that are averaged in expectation 
with $\alpha = (1+r)/2$.
\end{proof}

The simple example of a single Euclidean projector onto an affine subspace ($I=\{1\}$, and 
$T_1$ the orthogonal projection onto an affine subspace) shows that the statement of 
Proposition \ref{thm:contraInExpec} fails without the assumption that $T$ is a contraction.

\begin{prop}\label{t:linear convergence affine feas}
Given $a\in \Rn$, $b\in \mathbb{R}$, define the hyperplane $H = \set{y}{\ip{a}{y}=b}$ and  
fix $\xbar\in H$.  Define the random mapping 
   $\mymap{T_{(\xi,\zeta)}}{\Rn}{\Rn}$ by 
     \begin{align*}
    T_{(\xi,\zeta)}x\equiv P_{H^{(\xi,\zeta)}_\xbar}x &= x - 
    \frac{\act{a+\xi,x-\xbar} - \zeta}{\norm{a+\xi}^{2}} (a+\xi)
  \end{align*}
  where $(\xi,\zeta) \in \Rn\times\mathbb{R}$  
  is a vector of independent random variables.  
  
  Algorithm \eqref{algo:RFI} with this random function initialized with any 
  $\Rn$-valued random variable $X_0$ with distribution 
  $\mu^0\in\mathscr{P}(\Rn)$ converges to an invariant distribution whenever
  this exists.  
   
  If $(\xi,\zeta)$ satisfy 
\begin{subequations}\label{eq:noise assump}
  \begin{align}
    d&:= \mathbb{E} \left[
      \frac{(b+\zeta)^{2}}{\norm{a + \xi}^{2}}\right] < \infty, 
      \label{eq:noise assump zeta}\\
    c&:= \inf_{\substack{z \in \mathbb{S}}}
    \mathbb{E}\left[\frac{\act{a+\xi, z}^{2}}{\|a+\xi\|^2} \right] > 0
    \label{eq:noise assump xi}
  \end{align}
  where $\mathbb{S}$ is the set of unit vectors in $\Rn$
\end{subequations}
 then the Markov operator $\Pcal$ generated by $T_{(\xi,\zeta)}$ possesses a 
 unique invariant distribution and 
 Algorithm \eqref{algo:RFI} initialized with any $X_0$ with distribution 
   $\mu^0\in\mathscr{P}(\Rn)$ converges linearly to this distribution.  
\end{prop}

\begin{proof}
  Each mapping $T_{(\xi,\zeta)}$ is the orthogonal projector onto the 
  hyperplane $H^{(\xi,\zeta)}_\xbar$, and so 
  is averaged with constant $\alpha=1/2$.   
  Without regard to 
  the assumptions on the noise, based solely on 
  Theorem \ref{thm:a-firm convergence Rn} we conclude that 
  the iteration converges to a point in $\inv\mathcal{P}$ whenever this is nonempty. 

Existence of invariant distributions follows from the assumptions on the 
noise which impliy that $T_{(\xi,\zeta)}$ is actually a contraction in expectation.  
To see this, an elementary calculation shows that 
\begin{align*}
    \norm{T_{(\xi,\zeta)}x - T_{(\xi,\zeta)}y}^{2} = 
    \left(1-\cos^{2}\left(\tfrac{a+\xi}{\|a+\xi\|},\tfrac{x-y}{\|x-y\|}\right)\right) \norm{x-y}^{2}.
  \end{align*}
Taking the expectation over $(\xi, \zeta)$ yields 
  \begin{align*}
    \mathbb{E}\left[\norm{T_{(\xi,\zeta)} x - T_{(\xi,\zeta)}
        y}^{2}\right] \le (1-c) \norm{x-y}^{2}.
  \end{align*}
  From Proposition \ref{thm:contraInExpec} we get that there exists a
  {\em unique} invariant measure $\pi_0$ for $\mathcal{P}$ (even $\pi_0 \in
  \mathscr{P}_{2}$) and that it satisfies
  \begin{align*}
    W_{2}^{2}(\mu\mathcal{P}^{k},\pi_0) \le (1-c)^{k}
    W_{2}^{2}(\mu,\pi_0).
  \end{align*}
  Convergence is therefore linear.  
  \end{proof}
  
  Note that the noise satisfying \eqref{eq:noise assump} depends implicitly on the 
  point $\xbar$, which 
  will determine the concentration of the invariant distribution of the Markov operator.  This 
  corresponds to the fact that the exact projection, while unique, depends on the 
  point being projected.  One would expect the invariant distribution to be concentrated
  on the exact projection.  
  
  Extending this model to finitely many distorted
  affine subspaces as illustrated in Fig. \ref{fig:Axb} 
  (i.e.\ we are given $m$ normal vectors $a_{1},
  \ldots, a_{m} \in \Ecal$ and displacement vectors $b_{1},
  \ldots, b_{m}$) yields a stochastic 
  version of cyclic projections \eqref{eq:CP} which converges linearly 
  (geometrically) in the Wasserstein metric to a unique invariant measure for 
  the given noise model. 
  
  Indeed, for a collection of (not necessarily distinct) points 
  $\xbar_j\in H_j\equiv \set{y}{\ip{a_j}{y}=b_j}$ ($j=1,2,\dots,m$) denote  by 
  $P^{j}_{(\xi_j,\zeta_j)}$ the  exact projection onto the
  $j$-th random affine subspace centered on $\xbar_j$, i.e.\
  \begin{align*}
    P^{j}_{(\xi_j,\zeta_j)} x= x - \frac{\act{a_{j}+\xi_{j},x-\xbar_j} -
   \zeta_{j}}{\norm{a_{j}+\xi_{j}}^{2}} (a_{j}+\xi_{j}),
  \end{align*}
  where $(\xi_{i})_{i=1}^{m}$ and $(\zeta_{i})_{i=1}^{m}$ are
  i.i.d.\ and $(\xi_{i}) \indep (\zeta_{i})$. The stochastic cyclic projection
  mapping is 
\begin{align*}
    T_{(\xi,\zeta)} x = P^{m}_{(\xi_m,\zeta_m)} \circ \ldots 
    \circ P^{1}_{(\xi_1,\zeta_1)} x, \qquad x \in
    \Ecal
  \end{align*}
  where 
  $(\xi, \zeta) = ((\xi_m, \xi_{m-1}, \dots,\xi_1),(\zeta_m, \zeta_{m-1},\dots,\zeta_1))$.
Following the same pattern of proof as Proposition \ref{t:linear convergence affine feas} 
we see that $T_{(\xi,\zeta)}$
is a contraction in expectation:
  \begin{align*}
    \mathbb{E}\left[ \norm{T_{(\xi,\zeta)} x - T_{(\xi,\zeta)}
        y}^{2}\right] \le (1-c)^{m} \norm{x-y}^{2}
  \end{align*}
  where 
  \begin{align}
    c&:= \min_{j=1,\dots,m}\inf_{\substack{z \in \mathbb{S}}}
    \mathbb{E}\left[ \frac{\act{a_j+\xi_j,z}^{2}}%
    {\norm{a_j + \xi_j}^{2}} \right] > 0.
    \label{eq:noise assump xi2}
  \end{align}

  Hence, there exists a unique invariant measure and $(\mu\mathcal{P}^{k})$ 
  converges geometrically to it in the $W_{2}$ metric. Note that there is no 
  assumption of summability of the errors.  In fact, for the random function 
  iteration based on the usual additive noise model
  it can be shown that the Markov operator does not possess invariant distributions. 
  The assumption of summable errors in this case is tantamount to an assumption of no noise.  

\appendix

\section{Appendix}
\label{sec:toolbox}

\begin{prop}[Convergence with subsequences]\label{thm:cvg_subsequences}
  Let $(G,d)$ be a metric space. Let $(x_{k})$ be a sequence on $G$
  with the property that any subsequence has a convergent subsequence
  with the same limit $x\in G$. Then $x_{k} \to x$.
\end{prop}
\begin{proof}
  Assume that $x_{k} \not \to x$, i.e.\ there exists $\epsilon>0$ such
  that for all $N \in \mathbb{N}$ there is $k=k(N) \ge N$ with $d(x_{k},x)
  \ge \epsilon$. But by assumption the subsequence $(x_{k(N)})_{N \in
    \mathbb{N}}$ has a convergent subsequence with limit $x$, which is
  a contradiction and hence the assumption is false.
\end{proof}
\begin{rem}
  In a compact metric space, it is enough, that all cluster points are
  the same, because then every subsequence has a convergent
  subsequence.
\end{rem}

\begin{lemma}
  Let $(G,d)$ be a separable complete metric space and let the metric 
$d_{\times}$ on  $G\times G$ satisfy
  \begin{align}\label{eq:prodMetricCondition}
    d_{\times}\left( \icol{x_{k}\\ y_{k}}, \icol{x \\ y}\right)
    \to 0  && \Leftrightarrow && d(x_{k},x) \to 0 \quad\text{ and }\quad
    d(y_{k},y) \to 0.
  \end{align}
Then $\mathcal{B}(G\times G) = \mathcal{B}(G) \otimes
  \mathcal{B}(G)$.
\end{lemma}
\begin{proof}
  First we note that for $A,B \subset G$ it holds that $A\times B$ is
  closed in $(G\times G,d_{\times})$ if and only if $A,B$ are closed
  in $(G,d)$ by \eqref{eq:prodMetricCondition}. Since the
  $\sigma$-algebra $\mathcal{B}(G) \otimes \mathcal{B}(G)$ is
  generated by the family $\mathcal{A}:=\mysetc{A_{1}\times
    A_{2}}{A_{1},A_{2}\subset G \text{ closed}}$. One has 
    $\mathcal{B}(G\times G) \supset \mathcal{B}(G) \otimes
  \mathcal{B}(G)$

  For the other direction, note
  that any metric $d_{\times}$ with the property
  \eqref{eq:prodMetricCondition} has the same open and closed sets. If
  $A$ is closed in $(G\times G,d_{\times})$ and $\tilde d_{\times}$ is
  another metric on $G\times G$ satisfying
  \eqref{eq:prodMetricCondition}, then for $(a_{k},b_{k})\in A$ with
  $(a_{k},b_{k}) \to (a,b) \in G\times G$ w.r.t.\ $\tilde d_{\times}$
  it holds that $d(a_{k},a)\to 0$ and $d(b_{k},b)\to 0$ and hence $
  d_{\times}( (a_{k},b_{k}),(a,b)) \to 0$ as $k\to\infty$, i.e.\ $(a,b)\in A$, so $A$
  is closed in $(G\times G,\tilde d_{\times})$. It follows that all
  open sets in $(G\times G, d_{\times})$ are the same for any metric
 that satisfies \eqref{eq:prodMetricCondition}. 
 So, without loss of generality,  let 
  \begin{align}\label{eq:productMaxNorm}
    d_{\times}\left( \icol{x_{1}\\y_{1}}, \icol{x_{2}\\ y_{2}} \right) &=
    \max(d(x_{1},x_{2}),d(y_{1},y_{2})).
  \end{align}
Moreover, 
  separability of $G\times G$ yields that any open set is the
  countable union of balls: there exists a sequence $(u_{k})$
 on $U$ that is dense for $U \subset G\times G$ open. We can find a sequence of 
constants  $\epsilon_{k}>0$ with $\bigcup _{k} \mathbb{B}(u_{k},\epsilon_{k})
  \subset U$. If there exists $x \in U$, which is not covered by any
  ball, then we may enlarge a ball, so that $x$ is covered: since
  there exists $\epsilon>0$ with $\mathbb{B}(x,\epsilon) \subset U$
  and there exists $m \in \mathbb{N}$ with $d(x,u_{m})< \epsilon/2$ by
  denseness, we may set $\epsilon_{m} = \epsilon/2$ to get $x \in
  \mathbb{B}(u_{m},\epsilon_{m}) \subset \mathbb{B}(x,\epsilon)\subset
  U$.  Now to continue the proof, let $d_{\times}$ be given by
  \eqref{eq:productMaxNorm}. Then for any open $U \subset G\times G$
  there exists a sequence $(u_{k})$ on $U$ and  a corresponding sequence of 
  positive constants $(\epsilon_{k})$ such that $U =
  \bigcup_{k} \mathbb{B}(u_{k},\epsilon_{k})$.  This together with 
  the fact that 
  \begin{align*}
    \mathbb{B}(u_{k},\epsilon_{k}) = \mathbb{B}(u_{k,1},\epsilon_{k})
    \times \mathbb{B}(u_{k,2},\epsilon_{k}) \in \mathcal{B}(G) \otimes
    \mathcal{B}(G) \quad (u_{k}=(u_{k,1},u_{k,2}) \in G\times G) 
  \end{align*}
yields
      $\mathcal{B}(G\times G) \subset \mathcal{B}(G) \otimes
  \mathcal{B}(G),$
 which establishes equality of the $\sigma$-algebras.
\end{proof}

\begin{lemma}[couplings]\label{lemma:couplings}
  Let $G$ be a Polish space and let $\mu,\nu \in
  \mathscr{P}(G)$. Let $\gamma \in C(\mu,\nu)$, where
  \begin{align}
    \label{eq:couplingsDef}
    C(\mu,\nu) := \mysetc{\gamma \in \mathscr{P}(G\times G)}{ \gamma(A
      \times G) = \mu(A), \, \gamma(G\times A) = \nu(A) \quad \forall A
      \in \mathcal{B}(G)},
  \end{align}
  then
  \begin{enumerate}[(i)]
  \item\label{item:couplings1} $\Supp \gamma \subset \Supp \mu \times \Supp 
\nu$,
  \item\label{item:couplings2} $\overline{\mysetc{x}{(x,y) \in \Supp \gamma}} = 
\Supp \mu$.
  \end{enumerate}
\end{lemma}
\begin{proof}
  We let the product space be equipped with the metric in
  \eqref{eq:productMaxNorm} (constituting a separable complete metric space 
since 
  $G$ is Polish).
  \begin{enumerate}[(i)]
  \item Suppose $(x,y) \in \Supp \gamma$ and let $\epsilon>0$, then
    \begin{align*}
      \mu(\mathbb{B}(x,\epsilon)) = \gamma (\mathbb{B}(x,\epsilon)
      \times G) \ge \gamma(\mathbb{B}(x,\epsilon)\times
      \mathbb{B}(y,\epsilon)) = \gamma(\mathbb{B}( (x,y),\epsilon))>0.
    \end{align*}
    Analogously, we have $\nu(\mathbb{B}(y,\epsilon))>0$. So
    $(x,y)\in\Supp \mu \times \Supp \nu$.
  \item Suppose $x \in \Supp \mu$, then $\gamma(\mathbb{B}(x,\epsilon)
    \times G) >0$ for all $\epsilon>0$.  Since $G$ is Polish, the support of the measure is 
    nonempty whenever the measure is nonzero, and (again, since $G$ is Polish) 
    the support of the measure is closed, 
    there either exists $y \in G$ with $(x,y) \in \Supp \gamma$ or
    there exists a sequence $(x_{k},y_{k})$  on $\Supp \gamma$ with
    $x_{k} \to x$ as $k\to \infty$. Hence the assertion follows. \qedhere
  \end{enumerate}
\end{proof}

\begin{lemma}[convergence in product space]\label{lemma:weakCVG_productSpace}
  Let $G$ be a Polish space and suppose $(\mu_{k}),(\nu_{k})
  \subset \mathscr{P}(G)$ are tight sequences. Let $X_{k} \sim \mu
  _{k}$ and $Y_{k} \sim \nu_{k}$ and denote by $\gamma_{k} =
  \mathcal{L}( (X_{k},Y_{k}))$ the joint law of $X_{k}$ and
  $Y_{k}$. Then $(\gamma_{k})$ is tight.\\
  If furthermore, $\mu_{k} \to \mu \in \mathscr{P}(G)$ and $\nu_{k}
  \to \nu \in \mathscr{P}(G)$, then cluster points of $(\gamma_{k})$
  are in $C(\mu,\nu)$, where the set of couplings $C(\mu,\nu)$ is
  defined in \eqref{eq:couplingsDef} in Lemma \ref{lemma:couplings}.
\end{lemma}
\begin{proof}
  By tightness
  of $(\mu_{k})$ and $(\nu_{k})$, there exists for any $\epsilon> 0$ a
  compact set $K \subset G$ with $\mu_{k}(G\setminus K) < \epsilon/2$
  and $\nu_{k}(G\setminus K) < \epsilon/2$ for all $n \in \mathbb{N}$,
  so also
  \begin{align*}
    \gamma_{k}(G\times G \setminus K\times K) &\le \gamma_{k}(
    (G\setminus K)\times G) + \gamma_{k}( G \times (G\setminus K))
    \\ &= \mu_{k}(G\setminus K) + \nu_{k}(G\setminus K) \\ &< \epsilon
  \end{align*}
  for all $k \in \mathbb{N}$, implying tightness of $(\gamma_{k})$.
  By Prokhorov's Theorem \cite{Billingsley}, every subsequence of $(\gamma_{k})$ has
  a convergent subsequence $\gamma_{k_{j}} \to \gamma$ as
  $j\to\infty$ where $\gamma \in \mathscr{P}(G \times G)$. 
  
  It remains to show that $\gamma \in C(\mu,\nu)$.  Indeed,  since
  for every $f \in C_{b}(G \times G)$ we have $\gamma_{n_{k}} f \to
  \gamma f$, we can choose $f(x,y) = g(x) \1_{G}(y)$ with $g \in
  C_{b}(G)$.  Also,  
  \begin{align*}
    \mu g \leftarrow \mu_{n_{k}} g = \gamma_{n_{k}} f \to \gamma f =
    \gamma(\cdot\times G) g,
  \end{align*}
which  implies the equality $\mu = \gamma(\cdot\times G)$. Similarly
  $\nu = \gamma(G\times \cdot)$ and hence $\gamma \in C(\mu,\nu)$.
\end{proof}

\begin{lemma}[properties of the Prokhorov-L\`evy 
distance]\label{lemma:prokhorovDist_properties}
  Let $(G,d)$ be a separable complete metric space. 
  \begin{enumerate}[(i)]
  \item\label{item:prokLeviRep} The Prokhorov-L\`evy distance (Definition \ref{d:PL})
  has the representation
    \begin{align*}
      d_{P}(\mu,\nu) = \inf\mysetc{\epsilon>0}{ \inf_{\mathcal{L}(X,Y)
          \in C(\mu,\nu)} \mathbb{P}(d(X,Y) > \epsilon) \le \epsilon},
    \end{align*}
    where the set of couplings $C(\mu,\nu)$ is defined in
    \eqref{eq:couplingsDef} in Lemma \ref{lemma:couplings}.  Furthermore,
    the inner infimum for fixed $\epsilon>0$ is attained and the outer
    infimum is also attained.
  \item $d_{P}(\mu,\nu) \in [0,1]$.
  \item $d_{P}$ metrizes convergence in distribution, i.e.\ 
for $\mu_{k}, \mu
    \in \mathscr{P}(G)$, $k \in \mathbb{N}$ the sequence $\mu_{k}$ converges  to 
    $\mu$ in distribution if and only if $d_{P}(\mu_{k},\mu) \to 0$ 
as $k \to \infty$.
  \item $(\mathscr{P}(G), d_{P})$ is a separable complete metric space.
  \item\label{item:prokhorov5} For $\mu_{j},\nu_{j} \in \mathscr{P}(G)$ and 
  $\lambda_{j} \in [0,1]$,
    $j=1, \ldots, m$ with $\sum_{j=1}^{m}\lambda_{j} = 1$ we have 
    \begin{align*}
      d_{P}(\sum_{j}\lambda_{j}\mu_{j}, \sum_{j} \lambda_{j}\nu_{j}) \le 
\max_{j}
      d_{P}(\mu_{j}, \nu_{j}).
    \end{align*}
  \end{enumerate}
\end{lemma}
\begin{proof}
  \begin{enumerate}[(i)]
  \item See \cite[Corollary to Theorem 11]{Strassen65} for the first assertion. To
    see that the infimum is attained, let $\gamma_{k} \in C(\mu,\nu)$
    be a minimizing sequence, i.e.\ for $(X_{k},Y_{k}) \sim
    \gamma_{k}$ it holds that  $\mathbb{P}(d(X_{k},Y_{k}) > \epsilon) =
    \gamma_{k}(U_{\epsilon}) \to \inf_{(X,Y) \in C(\mu,\nu)}
    \mathbb{P}(d(X,Y) >\epsilon)$, where
    $U_{\epsilon}:=\mysetc{(x,y)}{d(x,y) > \epsilon} \subset G \times
    G$ is open. The sequence $(\gamma_{k})$ is tight and for a
    cluster point $\gamma$ we have $\gamma \in C(\mu,\nu)$ by
    Lemma \ref{lemma:weakCVG_productSpace}. From \cite[Theorem
    36.1]{Parthasarathy} it follows that $\gamma(U_{\epsilon}) \le
    \liminf_{j} \gamma_{k_{j}}(U_{\epsilon})$.\\
    To see, that the outer infimum is attained, let $(\epsilon_{k})$
    be a minimizing sequence, chosen to be monotonically nonincreasing
    with limit $\epsilon \ge 0$. One has that $U_{\epsilon} =
    \bigcup_{k} U_{\epsilon_{k}}$ where $U_{\epsilon_{k}} \supset
    U_{\epsilon_{k+1}}$ and hence $\gamma(U_{\epsilon}) = \lim_{k}
    \gamma(U_{\epsilon_{k}}) \le \lim_{k} \epsilon_{k} = \epsilon$.
  \item Clear by \eqref{item:prokLeviRep}.
  \item See \cite{Billingsley}.
  \item See \cite[Lemma 1.4]{Prokh56}.
  \item If $\epsilon>0$ is such that $\mu_{j}(A) \le
    \nu_{j}(\mathbb{B}(A,\epsilon))+\epsilon$ and $\nu_{j}(A)\le
    \mu_{j}(\mathbb{B}(A,\epsilon))+\epsilon$ for all $j =1,\ldots,m$ and all
    $A \in \mathcal{B}(G)$, then also $\sum_{j} \lambda_{j} \mu_{j}(A)
    \le \sum_{j} \lambda_{j} \nu_{j}(\mathbb{B}(A,\epsilon)) +\epsilon$ as well
    as $\sum_{j} \lambda_{j} \nu_{j}(A) \le \sum_{j} \lambda_{j}
    \mu_{j}(\mathbb{B}(A,\epsilon)) +\epsilon$.
  \end{enumerate}
\end{proof}



\begin{lemma}[properties of the Wasserstein metric]\label{lemma:WassersteinMetric_prop}
Recall  $\mathscr{P}_{p}(G)$ and $W_{p}$ from Definition \ref{d:PL}.
  \begin{enumerate}[(i)]
  \item The representation of $\mathscr{P}_{p}(G)$ is independent of
    $x$ and for $\mu,\nu \in \mathscr{P}_{p}(G)$ the 
    distance $W_{p}(\mu,\nu)$ is finite.
  \item\label{lemma:WassersteinMetric_prop ii} 
  The distance $W_{p}(\mu,\nu)$ is attained when it is finite. 
  \item\label{lemma:WassersteinMetric_prop iii} 
  The metric space $(\mathscr{P}_{p}(G),W_{p}(G))$ is complete and separable.
  \item\label{lemma:WassersteinMetric_prop iv} 
  If $W_{p}(\mu_{k},\mu) \to 0$ as $k\to\infty$ for the sequence 
  $(\mu_{k})$ on 
  $\mathscr{P}(G)$, then $\mu_{k} \to \mu$ as $k\to\infty$.
  \end{enumerate}
\end{lemma}
\begin{proof}
  \begin{enumerate}[(i)]
  \item See \cite[Remark after Definition 6.4]{Villani2008}.
  \item From Lemma \ref{lemma:weakCVG_productSpace} we know that a 
  minimizing sequence $(\gamma_{k})$ for $W_{p}(\mu,\nu)$  is tight and hence there is a
    cluster point $\gamma \in C(\mu,\nu)$. By continuity of the metric
    $d$ it follows that $d$ is lsc and bounded from
    below and from \cite[Theorem 9.1.5]{stroock2010probability}
    it follows that $\gamma d \le \liminf_{j} \gamma_{k_{j}} d =
    W_{p}(\mu,\nu)$.
  \item See \cite[Theorem 6.9]{Villani2008}.
  \item See \cite[Theorem 6.18]{Villani2008}.
  \end{enumerate}
\end{proof}
Note that the converse to 
Lemma \ref{lemma:WassersteinMetric_prop}
\eqref{lemma:WassersteinMetric_prop iv}
does not hold.  

\end{document}